\documentclass[12pt,reqno,twoside]{amsart}
\usepackage{amsmath, amssymb,latexsym}
\usepackage[all,cmtip]{xy}

\newtheorem{thm}{Theorem}[section]
\newtheorem{lem}[thm]{Lemma}
\newtheorem{prop}[thm]{Proposition}
\newtheorem{cor}[thm]{Corollary}
\newtheorem{defn}[thm]{Definition}

\newtheorem{question}{Question}
\newtheorem{obs}[thm]{Observation}

\numberwithin{equation}{section}

\newcommand{\dom}{{\text{dom}}}
\newcommand{\ran}{{\text{ran}}}

\newcommand{\N}{{\mathbb{N}}}

\begin{document}
\title[Density-1 Degrees]{The Generic Degrees of Density-1 Sets, and a Characterization of the Hyperarithmetic Reals}
\author{Gregory Igusa}

\maketitle

\begin{abstract}

 A generic computation of a subset $A$ of $\mathbb{N}$ is a computation which correctly computes most of the bits of $A$, but which potentially does not halt on all inputs. The motivation for this concept is derived from complexity theory, where it has been noticed that frequently, it is more important to know how difficult a type of problem is in the general case than how difficult it is in the worst case. When we study this concept from a recursion theoretic point of view, to create a transitive relationship, we are forced to consider oracles that sometimes fail to give answers when asked questions. Unfortunately, this makes working in the generic degrees quite difficult. Indeed, we show that generic reduction is $\mathbf\Pi^1_1-$complete. To help avoid this difficulty, we work with the generic degrees of density-1 reals. We demonstrate how an understanding of these degrees leads to a greater understanding of the overall structure of the generic degrees, and we also use these density-1 sets to provide a new a characterization of the hyperartithmetical Turing degrees.
   
\end{abstract}

\section{Introduction}

In complexity theory, there has been recent work attempting to rigorously understand and study the phenomenon in which a problem might be known to have a very high complexity in the traditional sense, and yet still be very easy to solve in practice. To this end, distinctions are made between the worst-case complexity of the problem, which is the usual way to measure the complexity of a problem, the average-case complexity of the problem \cite{Gur}, which measures the expected amount of time to solve the problem, and the generic-case complexity of the problem \cite{Kap}, a measure of how complex the majority of the instances of the problem are.

The study of generic-case complexity has led to the interesting realization that it is sometimes possible to find the generic-case complexity of a problem that is not even solvable. For instance, the word problem for Boone's group is known to be unsolvable, yet it can be shown to be generically linear time solvable \cite{Kap}. This sort of behavior allows for a complexity theoretic analysis of questions which had previously been outside of the scope of complexity theory. Simultaneously, however, it calls to light the recursion theoretic question: what can be said about the generically computable sets, and about generic computation in general?

Following the notation of Jockusch and Schupp \cite{stuff} we make the following definitions:

\begin{defn}\label{real}
\rm Let $A$ be a subset of the natural numbers. Then $A$ has \it{density 1} \rm if the limit of the densities of its initial segments is 1, or in other words, if $\lim_{n\rightarrow\infty}\frac{|{A\cap n}|}{n}=1$. In this case, we will frequently say that $A$ is \it{density-1.} \rm
\end{defn}

The notation in this paper follows the heuristics of \cite{Soa}, but notation will be defined as it is introduced. A subset of the natural numbers is often referred to as a real. Note that the intersection of two reals is density-1 if and only if each of the reals is density-1. (For any $\epsilon>0$, once the densities of the initial segments of each of the reals is $>1-\frac\epsilon2$, the density of their intersection will be $>1-\epsilon$.) We will sometimes use function notation for reals, in which case we say that $A(n)=1$ if $n\in A$, and $A(n)=0$ if $n\notin A$. In longer proofs, parenthetical comments will frequently be used to provide short proofs of claims in order to help illustrate the structure of the proofs.

\begin{defn}
\rm A real $A$ is \it{generically computable} \rm if there exists a partial recursive function $\varphi$ with the following properties:

\begin{itemize}
\item
$\dom(\varphi)$ is density-1, 
\item
$\ran(\varphi)\subseteq\lbrace 0,1\rbrace$,
\item
$\varphi(n)=A(n)$, for all $n\in\dom(\varphi)$.

\end{itemize}

\end{defn}

Note that this is very different from the following concept.

\begin{defn}
\rm A real $A$ is \it{coarsely computable} \rm if there exists a total recursive function $\varphi$, whose range is contained in $\lbrace 0,1\rbrace$ such that 
$\lbrace n\, |\, \varphi(n)=A(n)\rbrace$ is density-1.

\end{defn}

Thus, a generic computation is a computation which never makes mistakes, but which occasionally does not give answers, while a coarse computation is a computation which always gives answers, but sometimes makes mistakes. Neither generic computability nor coarse computability implies the other \cite{stuff}. The focus of this paper will be on generic computation (we remind the reader that the motivation for generic computation is algorithms that, in practice, run faster than they otherwise should be able to, not algorithms that take shortcuts and are occasionally inaccurate in order to get answers more quickly.) Coarse computation is only presented here to disambiguate, and also because we will find that the coarsely computable reals have a number of interesting properties from the point of view of generic computation.

Now, we wish to work our way up to generic degrees, and for this reason, we first present relativized generic computation.

\begin{defn}
\rm For reals $A$ and $B$, $A$ is \it{generically B-computable} \rm if $A$ is generically computable using $B$ as an oracle. In this case, we frequently say B \it{generically computes} \rm A.
\end{defn}

Notice, however, that this relativized notion of generic computation is very far from transitive, since information can be ``hidden" in a real in a way that causes it to have a large amount of computing power.

\begin{obs}

There exist reals $A$, $B$, and $C$ such that $B$ generically computes $A$, $C$ generically computes $B$, but $C$ does not generically compute $A$.

\end{obs}

\begin{proof}

Let $A$ be any real that is not generically computable. (There exists such a real because every $\varphi$ is a generic computation of at most Lebesgue measure zero many reals, and there are only countably many partial recursive functions.)

Let $B$ be the real such that $2^n\in B\Leftrightarrow n\in A$ (and $m\notin B$ if $m$ is not a power of 2.)

Let $C=0$, the empty set.

Then, $B$ generically computes $A$ because $B$ computes $A$, and a computation is also a generic computation. Also, $C$ generically computes $B$ by the algorithm $\varphi$ where $\varphi(m)=0$ if $m$ is not a power of 2. ($\varphi$ does not reference its oracle, and also $\varphi(m)$ does not halt if $m$ is a power of 2.)

Finally, $C$ does not generically compute $A$, because $A$ is not generically computable.

\end{proof}

Now, we introduce generic reduction, which is a notion of relative generic computation that has been modified to make it transitive. This will allow us to discuss the degree structure of the generic degrees, and will give us a deeper understanding of what it means to generically compute something, since now the difficulty in generically computing a real will be precisely measured by its generic computation power. We will also see in Observation \ref{compisred} that the entire theory of relative generic computation can be discussed within the structure of generic reduction, so we do not lose anything.

The basic idea of a generic reduction will be that a generic reduction from $A$ to $B$ is an algorithm that can use any generic oracle for $B$ to generically compute $A$. A generic oracle for $B$ is an oracle that does not always respond when asked a question, but that responds (always correctly) to density-1 many questions about $B$. For this reason, we first define what it means for a Turing machine to work with a partial oracle.

\begin{defn}

\normalfont

Let $A$ be a real. Then a (time-dependent) \emph{partial oracle,} $(A)$, for $A$ is a set of ordered triples $\langle n,x,l\rangle$ such that:

$\exists l\big(\langle n,0,l\rangle\in (A)\big)\Longrightarrow n\notin A$,

$\exists l\big(\langle n,1,l\rangle\in (A)\big)\Longrightarrow n\in A$.

\end{defn}

The idea here is that $(A)$ is a partial oracle that, when asked a question about $A$, sometimes takes a while before it responds, and does not always respond. Thus, to ``ask" $(A)$ whether or not $n\in A$ is to search $(A)$ for some $x,l$ such that $\langle n,x,l\rangle\in(A)$. Here, $x$ is thought of as the answer that $(A)$ gives, and $l$ is thought of as the amount of time before it gives an answer.

If such $x,l$ exist for $n$, we say that the oracle halts on $n$, i.e. $(A)(n)\downarrow$, and that the output of $(A)$ on $n$ is $x$; otherwise, we say that it does not halt on $n$, or $(A)(n)\uparrow$. Also, after querying the oracle (initiating a search for some such $x,l$) our reductions will be able to do other things while waiting for the oracle to respond (while running the search in parallel to other processes.) The ``domain" of a partial oracle, written $\dom((A))$, is the set of $n$ such that $(A)(n)\downarrow$.

The words ``time-dependent" in the definition refer to the fact that the original definition \cite{stuff} is different in that it does not use an $l$ parameter, and the eventual definition of a generic reduction uses enumeration reductions rather than Turing reductions. The two can be proven to be equivalent in our setting \cite{I}. We use this definition so that generic reductions can be formalized using Turing reductions, rather than enumeration reductions. This choice streamlines the proofs of several of our theorems, but it will force us to prove Lemma \ref{moreismore}, which is obtained for free with the other definition.

\begin{defn}

\normalfont

Let $A$ be a real. Then a \emph{generic oracle,} $(A)$, for $A$ is a partial oracle for $A$ such that $\dom((A))$ is density-1.

\end{defn}

Note then the following.

\begin{obs}\label{oraclesarecomputations}

$B$ generically computes $A$ if and only if $B$ computes a generic oracle for $A$. 

\end{obs}

\begin{proof}
If $\varphi^B$ is a generic computation of $A$, then we can let $\langle n,x,l\rangle\in(A)$ if and only if $\varphi^B(n)$ halts in $l$ steps, with value $x$. Conversely, if $B$ computes $(A)$ for some generic oracle, $(A)$, for $A$, then $B$ can generically compute $A$ by $\varphi^B(n)=x$ where $x$ is the first $x$ found such that there exists an $l$ with $\langle n,x,l\rangle\in(A)$. This algorithm halts on density-1, and gives only correct answers, since $(A)$ is a generic oracle for $A$.

\end{proof}

\begin{defn}
\normalfont
Let $A$ and $B$ be reals. Then, $A$ is \emph{generically reducible} to $B$ if there exists a Turing functional $\varphi$ such that for every generic oracle, $(B)$, for $B$, $\varphi^{(B)}$ is a generic computation of $A$. In this case, we write $B\geq_gA$.

\end{defn}

Here, it is important to note that a generic reduction is given by a uniform way to go from a generic oracle for $B$ to a generic computation of $A$. (In other words, the same reduction $\varphi$ must work for every generic oracle $(B)$.) If we relax this condition, and allow $\varphi$ to depend on $(B)$, then we get another, provably distinct, transitive notion of reducibility \cite{I}. Non-uniform generic reduction is somewhat more difficult to work with, and in this paper, we will focus solely on uniform generic reduction.

Note that the proof of Observation \ref{oraclesarecomputations} is uniform, which allows us to conclude that $\geq_g$ is a transitive relationship:

\begin{lem}\label{transitivity}
If $C\geq_gB$ and $B\geq_gA$, then $C\geq_gA$
\end{lem}

\begin{proof}
Assume $C\geq_gB$ and $B\geq_gA$. Then, let $(C)$ be a generic oracle for $C$. By assumption, $(C)$ can be used uniformly to generically compute $B$. By the uniformity of the proof of  Observation \ref{oraclesarecomputations}, $(C)$ can also be used to uniformly compute a generic oracle, $(B)$, for $B$, and, by assumption, $(B)$ can be used uniformly to generically compute $A$.

(The specific $(B)$ that is computed will depend on $(C)$, but any $(B)$ is sufficient to generically compute $A$.)
\end{proof}

With transitivity, we may now define the generic degrees in the standard manner.

\begin{defn}
\normalfont
The \emph{generic degrees} are the equivalence classes of reals under the relation $(A\equiv_gB)\leftrightarrow (A\geq_gB\wedge A\leq_gB)$.

As usual, the pre-ordering $\leq_g$ on reals  induces a partial ordering $\leq_g$ on the generic degrees:

${\bf a}\geq_g{\bf b}\leftrightarrow A\geq_gB$ where $A$ and $B$ are elements of $\bf a$, and $\bf b$ respectively.

\end{defn}

Next, we show that the usual join operation induces a degree-theoretic join in the generic degrees.

\begin{lem}\label{oplus_g}

Let $A$ and $B$ be reals. Let $A\oplus B$ be the real given by

\noindent$(2n\in A\oplus B \leftrightarrow n\in A)$ and $(2n+1\in A\oplus B \leftrightarrow n\in B)$.

Then, $A\oplus B\geq_gA$, $A\oplus B\geq_gB$, and, $\forall C$, if $C\geq_gA$ and $C\geq_gB$ then $C\geq_gA\oplus B$.

\end{lem}

\begin{proof}
The basic idea of the proof is that a subset of $\N$ is density-1 if and only if its even elements are density-1 in the even numbers, and its odd elements are density-1 in the odd numbers. The proof of this fact is a straightforward, but somewhat long limit calculation, and will be omitted.

From this, we deduce that it is easy to go from a generic oracle for $A\oplus B$ to one for each of $A$, and $B$, and vice versa.

\end{proof}

In terms of understanding the generic degrees, one of the most important features is that the Turing degrees embed naturally into the generic degrees.

\begin{defn}
\normalfont
For any real $X$, $\mathcal{R}(X)$ is the real defined by:

$\mathcal{R}(X)=\lbrace 2^mk\, |\, m\in X,k\in\N, k \text{ is odd}\rbrace$. 
\end{defn}

In other words, $n\in\mathcal{R}(X)$ if and only if $m\in X$, where $2^m$ is the largest power of $2$ dividing $n$.

\begin{prop}\label{R}
The map $X\mapsto\mathcal{R}(X)$ induces an embedding from the Turing degrees to the generic degrees.
\end{prop}

The key idea of this proof is that every piece of information in $X$ is stretched out over a positive density set in $\mathcal{R}(X)$, and so any generic description of $\mathcal{R}(X)$ contains all of the information in $X$.

The proof of the embedding consists of a proof that having access to a generic oracle for $\mathcal{R}(X)$ is exactly the same as having access to an oracle for $X$, and likewise, that the task of generically computing $\mathcal{R}(X)$ is exactly as difficult as the task of computing $X$. This is the sense in which the embedding is ``natural."

\begin{proof}

If $Y\geq_TX$, then we may generically reduce $\mathcal{R}(X)$ to $\mathcal{R}(Y)$ by using a generic oracle for $\mathcal{R}(Y)$ to compute $Y$, and then using $Y$ to compute $X$, which we use to compute $\mathcal{R}(X)$.

Let $(B)$ be a generic oracle for $\mathcal{R}(Y)$. Then, we can use $(B)$ to compute $Y$ in the following way.

To determine if $m\in Y$ we search for a $k\in\mathbb N$ such that $(2k+1)2^m\in\mathcal{R}(Y)$. Our oracle, $(B)$ must eventually give an answer for some such $K$, because otherwise its domain would not be density-1. $(B)$ is not allowed to make any mistakes, and so that answer must also be the answer to whether or not $m\in Y$.

Thus, to generically compute $\mathcal{R}(X)$ from $(B)$, we use $(B)$ to compute $Y$, use $Y$ to compute $X$, and use $X$ to compute $\mathcal{R}(X)$. (Note that for our purposes, we only need to generically compute $\mathcal{R}(X)$, but computing $\mathcal{R}(X)$ is just as good.)

Conversely, if $\mathcal{R}(Y)\geq_g\mathcal{R}(X)$, then in particular, $Y$ generically computes $\mathcal{R}(X)$. This is because $Y$ can compute a generic oracle for $\mathcal{R}(Y)$ (namely, the oracle that halts everywhere), and any generic oracle $\mathcal{R}(Y)$ can generically compute $\mathcal{R}(X)$. So then, to use $Y$ to compute $X(m)$, we use our generic computation of $\mathcal{R}(X)$ to simultaneously attempt to compute, for each $k$, whether or not $(2k+1)2^m\in\mathcal{R}(X)$. The generic computation must halt on at least one of these, else the domain of the generic computation is not density-1. When we find some $k$ such that $\varphi^Y((2k+1)2^m)$ halts, then we use that output for our computation of $X(m)$.

\end{proof}

Note, in particular, that by using this embedding, we are able to study generic computation in terms of generic reduction.

\begin{obs}\label{compisred}

Let $A,B$ be reals. Then $\mathcal{R}(B)\geq_gA$ if and only if $B$ generically computes $A$.

\end{obs}

\begin{proof}

If $\mathcal{R}(B)\geq_gA$ then $B$ generically computes $A$ because $B$ can compute a generic oracle for itself, which can be used to generically compute $A$.

Conversely, any generic oracle for $\mathcal{R}(B)$ can (uniformly) compute $B$ (as in the first half of the proof of Proposition \ref{R}), and so, if $B$ generically computes $A$, then any generic oracle for $\mathcal{R}(B)$ can generically compute $A$ by computing $B$, and then using the generic computation of $A$ from $B$.

\end{proof}

\section{Density-1 Reals}

\subsection{Introduction}

Now, to understand the generic degrees, we attempt to understand them from two standpoints. First, we attempt to answer basic degree-theoretic questions about them such as: Are there any minimal degrees? Are there minimal pairs of degrees? Does every nonzero degree have another degree that it is incomparable to? Second, we use the embedding of the Turing degrees to attempt to understand the generic degrees in terms of their relationship with the Turing degrees: How far up and how far down do the Turing degrees embed in the generic degrees? Can we understand concepts within the generic degrees in terms of known concepts within the Turing degrees?

For the question of ``how far up do the Turing degrees go?" the answer is straightforward, which is that they go all the way up.

\begin{obs}

For any real $A$, $\mathcal{R}(A)\geq_gA$. Thus, the embedded Turing degrees are upward dense in the generic degrees.

\end{obs}

\begin{proof}
Any generic oracle for $\mathcal{R}(A)$ can be used uniformly to compute $A$, and so, in particular, to generically compute $A$.
\end{proof}

The study of how far down the Turing degrees go is, in some sense, the study of the quasi-minimal generic degrees.

\begin{defn}\label{quasiminimal}
\normalfont

A nontrivial generic degree, {\textbf b}, is \emph{quasi-minimal} if for every Turing degree {\textbf a}, if ${\textbf b}\geq_g\mathcal{R}({\textbf a})$, then ${\textbf a}=0$.

\end{defn}

In other words, a quasi-minimal generic degree is a generic degree that is not generically computable, but is also not above any non-zero embedded Turing degrees.

Quasi-minimal degrees will be discussed more thoroughly in Section 2.3, when we discuss minimal degrees and minimal pairs in the generic degrees.

One of the biggest goals of this paper is to discuss the generic degrees of the density-1 reals, and show how they relate to many questions concerning the generic degrees in general. To aid our discussion, we define density-1 generic degrees as generic degrees that have density-1 elements.

\begin{defn}
\normalfont

A generic degree ${\bf a}$ is \emph{density-1} if there is a density-1 real $A\in{\bf a}$.

\end{defn}

The original motivation for studying the density-1 reals is that they, in some sense, isolate the ``new" concept that needs to be addressed in the generic degrees: A generic oracle sometimes outputs ``1," sometimes outputs ``0," and sometimes does not halt. Furthermore, the ``information" in the generic oracle is somehow contained both in the outputs that it gives, and also in the places where it does not give outputs. The density-1 reals isolate this second form of information, since a generic oracle of a density-1 real can be assumed to never give any outputs other than ``1," and so can be thought of as having all of its information contained within its halting set.

The following lemma formalizes the concept that, from the point of view of generic reduction, the entire information content of a density-1 real can be captured using oracles and computations that only output ``yes" answers when they halt.

\begin{lem}\label{output1}

Let $A$ be a density-1 real, and $B$ be any real. Then, the following hold.

 $B\geq_gA$ if and only if there exists a Turing functional $\varphi$ which only outputs 1's (for any real $X$ and number $n$, if $\varphi^X(n)\downarrow$, then $\varphi^X(n)=1$), such that for any generic oracle $(B)$ for $B$, $\varphi^{(B)}$ is a generic computation of $A$.

Also, $A\geq_gB$ if and only if there exists a Turing functional $\varphi$ such that for any generic oracle $(A)$ for $A$, if $\dom((A))\subseteq A$, then $\varphi^{(A)}$ is a generic computation of $B$.

\end{lem}

The primary content of this lemma is that a generic oracle or generic computation of a density-1 real, $A$, can, in some sense, be thought of as an enumeration of a density-1 subset of $A$. This idea will be formalized more thoroughly in Lemmas \ref{d1join} and \ref{d1subset}.

\begin{proof}
Let $A$ be a density-1 real, and assume $B\geq_gA$.

Fix $\psi$ such that for any generic oracle $(B)$ for $B$, $\psi^{(B)}$ is a generic computation of $A$. If we modify $\psi$, and define $\varphi$ so that $\varphi^X$ behaves exactly the same way as $\psi^X$, except that $\varphi^X(n)$ diverges whenever $\psi^X(n)\neq1$, then $\varphi$ will also have the property that for any $(B)$, $\varphi^{(B)}$ is a generic computation of $A$, because the outputs that it gives will still be correct ($\varphi$'s outputs are a subset of $\psi$'s outputs), and the domains of the computations will still be density-1, since for each $(B)$, the domain of $\varphi^{(B)}$ will be the intersection of $A$ with the domain of $\psi^{(B)}$.

Conversely, if there exists a $\varphi$ as in the statement of the observation, then that $\varphi$ is a generic reduction from $A$ to $B$.

For the second part of the observation, assume there exists a $\varphi$ that generically computes $B$ from any generic oracle for $A$ which outputs only 1's. Then it can be modified to generically compute $B$ from any generic oracle $(A)$ for $A$ in the following manner. $A$ is density-1, and $(A)$ is a generic oracle, and so the intersection of $A$ with the domain of $(A)$ must be density-1. So now, we define $\psi$ so that $\psi^X(n)=\varphi^Y(n)$ where $Y=X\setminus\lbrace\langle n,0,l\rangle\,|\,n,l\in\N\rbrace$. (Intuitively, $\psi$ is the computation that mimics $\varphi$, except that it ignores any locations where its oracle outputs 0.) Then, $\psi$ will generically compute $B$ from any generic oracle for $A$, because it modifies that oracle to a generic oracle that outputs only 1's, and then uses $\varphi$.

Again, the converse is true by definition: If $A\geq_gB$, then there exists a $\varphi$ which generically computes $B$ from any generic oracle for $A$. In particular, $\varphi$ generically computes $B$ from any generic oracle for $A$ that outputs only 1's.

\end{proof}

Note that for this result, we needed to use the fact that the intersection of two density-1 reals is density-1. We provided a short proof of this after Definition \ref{real}, but here, we will state and prove this observation clearly, since it will be very relevant to much of our work with density-1 reals.

\begin{obs}\label{int}
Let $A$ and $B$ be reals. Then $A\cap B$ is density-1 if and only if both $A$ and $B$ are density-1.
\end{obs}

\begin{proof}
If $A\cap B$ is density-1, then each of $A$ and $B$ is density-1, because it is a superset of a density-1 set.

Conversely, fix some $\epsilon>0$. We must prove that there is some $m$ such that $\forall n>m\ \  \frac{|{(A\cap B)\cap n}|}{n}>1-\epsilon$.

If $A$ and $B$ are each density-1, then choose $m$ such that for every $n>m$, $\frac{|{A\cap n}|}{n}>1-\frac\epsilon2$ and $\frac{|{B\cap n}|}{n}>1-\frac\epsilon2$. This $m$ suffices for our purposes, since if anything is missing from $(A\cap B)\cap n$, then it is either missing from $A\cap n$, or from $B\cap n$.

\end{proof}

This allows us to prove a few results which will be useful to us, and will also help illustrate the manner in which density-1 reals are easier to work with than other reals.


\begin{lem}\label{d1sub}
Let $A$ and $B$ be density-1 reals. Assume $B\subseteq A$. Then $B\geq_gA$.
\end{lem}

\begin{proof}
Assume $B\subseteq A$.

Then $B\geq_gA$ via the algorithm $\varphi$ where $\varphi(n)=1$ if $n\in B$. This algorithm only gives correct outputs because $B\subseteq A$, and it halts on the intersection of $B$ with the domain of the generic oracle for $B$, which is density-1.

\end{proof}

\begin{lem}\label{d1join}
Let $A$ and $B$ be density-1 reals. Then $A\cap B\equiv_gA\oplus B$.

\end{lem}

\begin{proof}
To show that $A\cap B\geq_gA\oplus B$, we show that $A\cap B\geq_gA$, and that $A\cap B\geq_gB$. This is true by Lemma \ref{d1sub}.

Conversely, $A\oplus B\geq_gA\cap B$ by the algorithm $\varphi$ where $\varphi(n)=1$ if $2n$ and $2n+1$ are both in $A\oplus B$. (If we would like, we may also have $\varphi(n)=0$ if either $2n$ or $2n+1$ is not in $A\oplus B$, this would still be a correct generic reduction, but it is against the philosophy implicit in Lemma \ref{output1}.) If $(C)$ is a generic oracle for $A\oplus B$, then the domain of $\varphi^{(C)}$ is density-1 because it is the intersection of the domains of the natural generic computations of $A$ and of $B$ from $(C)$.

\end{proof}

\begin{lem}\label{d1subset}
Let {\textbf a} and {\textbf b} be density-1 generic degrees. Then ${\textbf b} \geq_g {\textbf a}$ if and only if $\exists B\in {\textbf b}\ \exists A\in {\textbf a}$ such that $B$ and $A$ are both density-1, and $B\subseteq A$.
\end{lem}

\begin{proof}

Assume $\exists B\in {\textbf b}\ \exists A\in {\textbf a}(B\subseteq A)$, with $B$ and $A$ both density-1. Then $B\geq_gA$ by Lemma \ref{d1join}.

Conversely, assume ${\textbf b} \geq_g {\textbf a}$, and let $B_0\in {\textbf b},A_0\in {\textbf a}$, with $B_0,A_0$ both density-1. Let $B=B_0\cap A_0$, and let $A=A_0$. Clearly, $B\subseteq A$. Also, by Lemma \ref{d1join}, $B\equiv_g A_0\oplus B_0$. Furthermore, $A_0\oplus B_0\equiv_g B_0$ because $B_0\geq_gA_0$. Thus, $B\in {\textbf b},A\in {\textbf a}$, and $B\subseteq A$.
\end{proof}


We use Lemma \ref{d1subset} to prove that the density-1 sets are dense in the generic degrees

\begin{prop}\label{denseisdense}
Let {\textbf a} and {\textbf b} be density-1 generic degrees. Assume ${\textbf b} >_g {\textbf a}$. Then there exists a density-1 degree {\textbf c} such that ${\textbf b} >_g {\textbf c} >_g {\textbf a}$.

\end{prop}

Before we prove this, however, we require a technical result roughly saying that in generic reductions, we may assume that computations always give more outputs if they get more inputs.

\begin{lem}\label{moreismore}

Let $A,B$ be reals. Then $B\geq_gA$ if and only if there is a $\varphi$ such that the following hold.

\begin{itemize}
\item
For any generic oracle, $(B)$, for $B$, $\varphi^{(B)}$ is a generic computation of $A$.
\item
For any $C$, if anything generically reduces to $C$ via $\varphi$, then for any partial oracles ${(C)_0}$ and $(C)_1$, if $\text{dom}({(C)_0})\subseteq\text{dom}({(C)_1})$, and ${(C)_1}\upharpoonright \text{dom}({(C)_0})={(C)_0}$ (with ${(C)_0}$ and ${(C)_1}$ regarded as partial functions), then $\text{dom}(\varphi^{(C)_0})\subseteq\text{dom}(\varphi^{(C)_1})$, and $\varphi^{(C)_1}\upharpoonright \text{dom}(\varphi^{(C)_0})=\varphi^{(C)_0}$.
 
\end{itemize}

\end{lem}

We call such a $\varphi$ a ``more-is-more" functional, since having more information from the oracle never results in computing fewer things from the output. (Note that in Jockusch and Schupp's definition of generic reduction \cite{stuff}, all generic reductions are via more-is-more functionals.)

\begin{proof}

Certainly, if there is a $\varphi$ as in the statement of the lemma, then $B\geq_gA$, because the first bulletpoint ensures that $\varphi$ is a generic reduction of $A$ to $B$.

Conversely, assume that $B\geq_gA$ via $\psi$. Then, define $\varphi^X$ as follows.

Think of $X$ as a partial oracle, and consider all partial oracles $Y$ that agree with $X$ (i.e., so that $Y$ does not give any outputs that $X$ does not give). For any such $Y$, if $\psi^Y(n)=x$ then $\varphi^X(n)=x$. (If there are multiple such $Y$s for which $\psi^Y(n)$ is defined, then use the first such $Y$ that is found.)

Note then that, first of all, if $X$ is a generic oracle for $B$, then $\varphi^X$ is a generic computation for $A$, because every $Y$ that is used will be a partial oracle for $B$, and any finite portion of a partial oracle for $B$ can be extended to a generic oracle for $B$, and so $\psi^Y$ cannot make any mistakes when generically computing $A$. Anything that $\psi^X$ would output, will also be outputted by $\varphi^X$, and so the domain of $\varphi^X$ will be density-1.

Also, $\varphi$ satisfies the ``more is more" requirement of the lemma, because if ${(C)_0}$ and $(C)_1$ are as in the lemma, then 
any $Y$ that agreed with $(C)_0$ would also agree with $(C)_1$,
and so $\varphi^{(C)_1}$ would halt anywhere that $\varphi^{(C)_0}$ would halt. If anything generically reduces to $C$ via $\varphi$, then in particular, $\varphi^{(C)_1}$ must agree with $\varphi^{(C)_0}$ anywhere that they both halt.

\end{proof}

Now we move on to prove Proposition \ref{denseisdense}

\begin{proof}
Use Lemma \ref{d1subset} to fix density-1 sets $B\subseteq A$ in {\textbf b} and {\textbf a} respectively. We build a real $C$ such that $B\subseteq C\subseteq A$. This guarantees that $C$ is density-1 and that $B\geq_g C\geq_g A$. The main difficulty in the construction will be ensuring that $C\ngeq_gB$, and $A\ngeq_gC$.

Note that it is not necessary to ensure that $B\geq_TC$, and indeed, this will probably not be the case.

The basic idea of the proof is that $C$ will alternate between copying $B$ until it forces one instance of $A$ not computing it, and copying $A$ until it forces one instance of it not computing $B$. By the end of the construction, there will be no $\varphi$ via which $A$ could compute $C$, or via which $C$ could compute $B$.

\vspace{5pt}

At stage $2e$, we have some finite approximation $\sigma_{2e}$ to $C$, and we wish to extend it to ensure that $C$ will not generically reduce to $A$ via $\varphi_e$. 

The first thing that we ask is whether $\varphi_e$ is a functional that only outputs 1's when it halts. If not, then we do not have to do anything, since $C$ will be density-1 by the end of the construction, so by Lemma \ref{output1}, if there is a reduction from $C$ to $A$, then there is a reduction that only outputs 1's.

Next, we ask whether it is true that for every generic oracle $(A)$, for A, $\varphi_e^{(A)}$ is a generic computation of $A$. If not, then we do not have to do anything, since $C$ will be a density-1 subset of $A$, so any generic computation of $C$ that only outputs 1's must also be a generic computation of $A$. 

If so, then there must be some number $n$, and some generic oracle, $(A)$, for $A$ such that $\varphi_e^{(A)}(n)=1$, but $n\notin B$. Otherwise, $A\geq_gB$ via $\varphi_e$. (We are assuming that $\varphi_e^{(A)}$ halts on density-1, and only outputs 1's, so the only thing that could prevent it from being a generic computation of $B$ is if it outputs a 1 when it is not supposed to.)

By the usual argument, there must be infinitely many such $n$ (because otherwise we could modify $\varphi_e$ to not halt on those $n$, which is not possible, because we know that $A\ngeq_gB$.) So we can choose one such $n$ that is larger than $|\sigma_{2e}|$, and extend our approximation to $C$ to be equal to $B$ up to and including that $n$.

This ensures that it is not true that $A\geq_gC$ via $\varphi_e$, since $n\notin C$, but $\varphi_e^{(A)}(n)=1$.

\vspace{5pt}

At stage $2e+1$, we have some finite approximation $\sigma_{2e+1}$ to $C$, and we wish to extend it so that $B$ does not generically reduce to $C$ via $\varphi_e$.

The first thing we ask is whether $\varphi_e$ is a more-is-more functional. If not, then we do not need to do anything, since, by Lemma \ref{moreismore}, if there is a reduction from $B$ to $C$, then there is a reduction via a more-is-more functional.

Next, we ask whether it is true that for every generic oracle $(B)$, for B, with dom$((B))\subseteq B$, $\varphi_e^{(B)}$ is a generic computation of $B$. If the answer is ``no," then we do not have to do anything: $C$ will contain $B$, so every such generic oracle for $B$ will, in fact, also be a generic oracle for $C$. Thus if any of the $\varphi_e^{(B)}$ is not a generic computation of $B$, then $B$ will definitely not reduce to $C$ via $\varphi_e$.

If the answer is ``yes," then it must be true that for some $n$, $\varphi_e^A(n)\neq B(n)$. Otherwise, $A\geq_gB$ because for any $(A)$, $\varphi_e^{(A)}$ never gives incorrect outputs about $B$. (Notice that any generic oracle for $A$ actually is a superset of some generic oracle for $B$, so since $\varphi_e$ is a more-is-more functional, for any $(A)$, the domain of $\varphi_e^{(A)}$ is density-1.)

Similarly, if we let $A_1$ be the real that agrees with $\sigma_{2e+1}$ up to $|\sigma_{2e+1}|$, and agrees with $A$ after that, then there must also be some  $n\notin B$, such that $\varphi_e^{A_1}A(n)=1$ because otherwise $A_1\geq_gB$. This is not possible, since $A_1\equiv_gA$. (Recall that $A_1$ is a finite modification of $A$.) We extend our approximation to $C$ to match $A$ until we have copied enough to make $\varphi_e^{C}(n)\neq B(n)$ for some $n$.

This ensures that it is not true that $C\geq_gB$ via $\varphi_e$.

\vspace{5pt}

Once we have completed $\omega$-many stages, we will have ensured that, for each $e$, if $\varphi_e$ only outputs 1's, then it does not witness $A\geq_gC$, and also, if $\varphi_e$ is a more-is-more reduction, then it does not witness $C\geq_gB$. Thus, by Lemmas \ref{output1} and \ref{moreismore}, $A\ngeq_gC$ and $C\ngeq_gB$. Hence, because $B\subseteq C\subseteq A$, we have that that $B>_g C>_g A$.

Let {\textbf c} be the generic degree of $C$.

\end{proof}

We can strengthen this result to split $A$ over $B$.

\begin{prop}
Let {\textbf a} and {\textbf b} be density-1 generic degrees. Assume ${\textbf b} >_g {\textbf a}$. Then there exist density-1 degrees {\textbf c} and {\textbf d} such that ${\textbf b} >_g {\textbf c} >_g {\textbf a}$, ${\textbf b} >_g {\textbf d} >_g {\textbf a}$, and ${\textbf c}\oplus{\textbf d}={\textbf b}$.

\end{prop}

\begin{proof}
The basic idea of this proof is that we will mimic the previous construction, but we will ensure that $C\cap D=B$, which will ensure that $C\oplus D\equiv_g B$ by Lemma \ref{d1join}. We may also, if we desire, ensure that $C\cup D=A$, which we will do, just for symmetry, but this does not ensure that ${\textbf a}$ is the infimum of {\textbf c} and {\textbf d} in the generic degrees.

We will assume familiarity with the proof of Proposition \ref{denseisdense}.

\vspace{5pt}

As before, fix density-1 sets $B\subseteq A$ in {\textbf b} and {\textbf a} respectively.

At stage $2e$, we have some finite approximations $\sigma_{2e}$ to $C$, and $\tau_{2e}$ to $D$, and we wish to extend them so that $C$ does not generically reduce to $A$ via $\varphi_e$, and so that $B$ does not generically reduce to $D$ via $\varphi_e$. (Notice here that we are satisfying requirement $2e$ for $C$, but requirement $2e+1$ for $D$.)

We accomplish this by having $C$ copy $B$ until for some $(A)$, $\varphi_e^{(A)}$ incorrectly computes some bit of $C$. We simultaneously have $D$ copy $A$ until for some $(D)$, $\varphi_e^{(D)}$ incorrectly computes some bit of $B$.

(If one of the strategies satisfies its objective before the other one does, then it continues to copy the set that it is copying until the other strategy has satisfied its own objective.) If one of the strategies does not need to act, then both strategies only wait for the one that needs to act to satisfy its objective. If neither strategy needs to act, then we simply move on to stage $2e+1$. (Again, this cannot be done uniformly, but that is fine, because we do not need either $C$ or $D$ to be computable from $B$, but just for them both to be generically computable from $B$. This is guaranteed since both $C$ and $D$ contain $B$.)

At stage $2e+1$, we have $C$ copy $A$ and $D$ copy $B$ in a manner analogous to what we did at stage $2e$.

At the end of the construction, we have that $C\cap D=B$,  and so $C\oplus D\equiv_g B$. We have also forced that $C\ngeq_gB, D\ngeq_gB, A\ngeq_gC,$ and $A\ngeq_gD$. So we may let {\textbf c}, and {\textbf d} be the generic degrees of $C$, and $D$, respectively.

\end{proof}

\subsection{Bounding Hyperarithmetic Reals}

Now, that we have established some techniques for working with density-1 degrees, we seek to prove Theorem \ref{hyp}, which says that for any real $A$, the Turing degree of $A$ is hyperarithmetic if and only if there is some density-1 $B$ such that $B\geq_g\mathcal{R}(A)$. For our purposes, the most useful characterization of the hyperarithmetic reals is that  is that the hyperarithmetic reals are precisely the reals that can be computed from any sufficiently fast growing function.

\begin{thm}\label{hyphasmodulus}
\normalfont{(Solovay)} \cite{Sol}

\emph{Let $A$ be a real. Then $A$ is hyperarithmetic if and only if there is a function $f$, and a Turing functional $\varphi$ such that for every function $g$ dominating $f$, $\varphi^g$ is a computation of $A$. In this case, we say that $f$ is a modulus of computation for $A$.}

\emph{(Here, $g$ dominates $f$ if and only if $\forall n$, $g(n)\geq f(n)$. In this case, we sometimes write $g\gg f$.)}

\end{thm}

Here, $A$ is hyperarithmetic if and only if there is a recursive ordinal $\alpha$ such that $A\leq_T0^{\alpha}$, the $\alpha$th jump of $0$. Equivalently, as shown by Kleene, $A$ is hyperarithmetic if and only if it is $\Delta^1_1$ (i.e., definable by both a $\Sigma^1_1$ formula and a $\Delta^1_1$ formula in the language of second order arithmetic.) See \cite{sacks} for a more thorough explanation of the subject. In this paper, however, we will only need the characterization of the hyperarithmetic reals in terms of fast growing functions, so we do not present the other two characterizations in detail, and Theorem \ref{hyphasmodulus} may be treated as a definition.

The easier direction of our argument will be showing that anything computable from a sufficiently fast growing function, $f$, can be coded into a density-1 set, $A$, in such a way that any generic oracle for $A$ can recover a function that dominates $f$.

\begin{prop}\label{fastgrowing}
Let $A$ be any hyperarithmetic real. Then there is a density-1 real $B$ such that $B\geq_g \mathcal{R(A)} $.

\end{prop}

The idea of this proof is that, if $f$ is a fast growing function, we may define $B$ so that $B$ is density-1, but the density of $B$ approaches 1 very slowly. If we do this, then any generic oracle, $(B)$, for $B$ will have its density approach 1 at most as quickly as $B$'s density does, and thus $(B)$ will be able to compute a function that dominates $f$. $(B)$ can then use that function to compute $A$.

\begin{proof}
By Theorem \ref{hyphasmodulus}, there is a function $f$, and a Turing functional $\varphi$ such that for any $g$, if $\forall n\, g(n)\geq f(n)$, then $\varphi^g$ is a computation of $A$.

Let $f$, $\varphi$ be as above. Replacing $f$ with a faster growing function if necessary, we may assume $f$ is an increasing function.

Let $B$ be any density-1 real such that for every $n$ and $m$, if $n< f(m)$, then $\frac{|{B\cap n}|}{n}<1-2^{-m}$. Such a real can be constructed by the simple algorithm, in which for each $n$, we decide whether $n$ goes into $B$ after determining $B\cap n$. We make this decision based on: $n$ goes into $B$ if and only if putting it in to $B$ would not cause $\frac{|{B\cap (n+1)}|}{n+1}$ to go above $1-2^{-m}$, where $m$ is the smallest $m$ such that $f(m)>n+1$. (Since $f$ is an increasing function, such an $m$ exists.)

When $B$ is defined in this manner, it will have density 1 because, as $n$ gets large, the restrictions on putting elements into $B$ get weaker. Thus, for any fixed value of $m$, we will eventually stop requiring that $\frac{|{B\cap (n+1)}|}{n+1}<1-2^{-m}$. Then, once $n$ becomes large enough, $\frac{|{B\cap (n+1)}|}{n+1}$ will be greater than $1-2^{-m}$. (We need $n$ to be sufficiently large that whatever happened earlier is negligible, and also that adding or removing $n$ from $B$ will not change $\frac{|{B\cap (n+1)}|}{n+1}$ by more than $2^{-(m+1)}$.)

Then, from any generic oracle $(B)$ for $B$, we can define a function $g$, where $g(m)$ is the first number $n$ such that we see that $\frac{|{B\cap\, \text{dom}((B))\cap n}|}{n}\geq1-2^{-m}$. For every $m$, $g(m)$ is defined because the density of $(B)$ must approach 1, so there must be some least $n_0$ such that for all $n_1\geq n_0$, $\frac{|{B\cap\, \text{dom}((B))\upharpoonright n_1}|}{n_1}>1-2^m$, and so there must be some $n\geq n_0$ where we see that $\frac{|{B\cap\, \text{dom}((B))\cap n}|}{n}\geq1-2^{-m}$. (Note that, since $(B)$ only provides an enumeration of its domain, the first such $n$ that we find might not be the smallest such $n$. Generic oracles sometimes take a long time to give their answers, so we slowly increase the values of $n$ that we check while waiting for the oracle to converge on all inputs less than $n$.)

Since $B\cap \text{dom}((B))\subseteq B$, we have that for every $m$, $g(m)\geq f(m)$, so we may compute $A$, and therefore compute (and generically compute) $\mathcal{R}(A)$ via $n\in A \leftrightarrow\varphi^g(n)=1$.

\end{proof}

The other direction of our argument is more subtle, because, as we will see in Proposition \ref{d1ismore}, there is no way to code an arbitrary density-1 real within a modulus of computation. Instead, we are forced to work both with our algorithm, and with the real being computed.

\begin{prop}\label{fastgrowing2}
Let $A$ be a real, and suppose $B$ is a density-1 real such that $B\geq_g \mathcal{R(A)}$. Then $A$ is hyperarithmetic.

\end{prop}

The basic idea of the construction is that we use $B$ to generate a function such that from any faster growing function, one can generate a binary tree of density-1 oracles that includes $B$. (The function generated is not the same as the function from Proposition \ref{fastgrowing}. This function's purpose is in some sense dual to the purpose of the previous function.)

Once we have this class of oracles, we have them engage in a process that can be visualized as voting in pairs until they can find one leader who is strong-willed enough to make them vote unanimously. The oracle $B$ is able to make them vote unanimously, so eventually they will find such a leader. Also, $B$ is not corruptible, so when they find such a leader, even if that leader is not $B$, that leader will have $B$'s support, and so in particular, that leader will lead the oracles to the correct conclusion.

\begin{proof}

By assumption, let $A$ be a real, $B$ a density-1 real, and $\varphi$ a Turing functional such that $B\geq_g \mathcal{R}(A)$ via $\varphi$.

By Theorem \ref{hyphasmodulus}, we must prove that there is a function $f$, and a Turing functional $\psi$ such that for any $g$, if $g>>f$, then $\psi^g$ is a computation of $A$.

Let $f$ be the function where $f(m)$ is the smallest number such that $\forall n> f(m),\ \frac{|{B\cap n}|}{n}>1-2^{-m}$. Then we claim that from any $g$ that dominates $f$, we can uniformly compute $A$.

The basic idea of the construction is that from any such $g$, we can get a lower bound on the rate at which the density of $B$ goes to 1. We then consider all density-1 oracles whose density goes to 1 at at least this fast. (We use $g$ to build a tree of possibilities for $B$ that includes $B$, but also only includes paths that are density-1.) Then, to get an answer from our collection of oracles we have them work together in pairs in a way that ensures that whatever answer we get is an answer that has been approved by $B$.

Let $g$ be a function such that for all $m$, $g(m)\geq f(m)$. Replacing $g$ by a faster growing function if necessary, we may assume that $g$ is an increasing function. (Replace $g$ with $\tilde g$, where $\tilde g(m)=\text{max}(g(m),\tilde g(m-1)+1)$.)

Define $T_g\subseteq2^{<\omega}$ to be the tree such that $\sigma\in T_g$ if and only if $\forall n,m$, if $n>g(m)$, and if $n\leq|\sigma|$, then $\frac{\#\lbrace k<n\, |\, \sigma(k)=1\rbrace}{n}>1-2^{-m}$. The important facts about $T_g$ are that every path through $T_g$ is density-1, $B$ is a path through $T_g$, and that $T_g$ is uniformly recursive in $g$.

The first fact holds because $g$ provides a lower bound on the rate at which the density of a path must go to 1, and the paths through $T_g$ all respect that lower bound. The second fact holds because, by the definition of $f$, $B$ is a path through $T_f$, and faster growing functions provide larger trees, not smaller ones. The third fact holds because the definition is uniform in $g$ (and the apparently unbounded quantifiers over $n$ and $m$ are bounded by $g(m)<n\leq|\sigma|$. Since $g$ is an increasing function, a bound on $g(m)$ is a bound on $m$.)

To determine whether $n\in A$, we search for a real $X_0$ such that for every $X$ that looks like it might be a path through $T$, if we let $Y_X=X_0\cap X$, and let $(Y_X)$ be the partial oracle for $Y_X$ that only halts on the elements of $Y_X$, (and that halts immediately if it halts,) then for some odd value of $k$, $\varphi^{(Y_X)}(2^nk)$ halts, and such that all of those computations (ranging over different reals $X$) halt and give the same answer. Then $n\in A$ if and only if that answer is ``1."

Such an $X_0$ exists because $B$ is such a real. (If we let $X_0=B$, then for every $X$, if $X$ is a path through $T$, then $(Y_X)$ is a generic oracle for $B$, and so $\varphi^{(Y_X)}$ must be a generic computation of $\mathcal{R(A)}$, and thus it must halt on $2^nk$ for some odd $k$. This is observed at some finite stage by the usual compactness argument --- Every $X$ must either at some finite height fall out of $T$, or at some finite stage, with some finite portion of itself, provide a correct answer by having $\varphi^{(Y_X)}(2^nk)$ halt. Note that since all of this happens at a finite stage, we do not need to search over all uncountably many potential values of $X_0$, but rather only over all potential finite initial segments of $X_0$, so the search for $X_0$ can be conducted effectively.)

Furthermore, the answer given when $X_0$ is found must be correct, because $B$ is a path through $T_g$. Thus is always one of the eligible values for $X$, so for any $X_0$, $(Y_B)$ is a generic oracle for $B$. Thus, $\varphi^{(Y_B)}$ cannot give any incorrect outputs for $\mathcal{R(A)}$.

Thus, once again, intuitively speaking, $B$ is smart enough to force every $X$ to give the correct vote, so a consensus must eventually be reached. No $X_0$ is able to force $B$ to vote incorrectly, so any reached consensus must be a correct one. We use $g$ only to build a population of density-1 sets that includes $B$ as a member.

\end{proof}

We may now conclude the main theorem of this section.

\begin{thm}\label{hyp}
A real $A$ is hyperarithmetic if and only if there is a density-1 real $B$ such that $B\geq_g \mathcal{R(A)}$.

\end{thm}

\begin{proof}
This follows directly from Propositions \ref{fastgrowing} and \ref{fastgrowing2}.

\end{proof}

We also mention a somewhat useful corollary that we get from the proof of Proposition \ref{fastgrowing2}.

\begin{cor}

Let $B$ be a density-1 real, and assume that there is a recursive lower bound on the rate at which the density of $B$ approaches 1. Then $B$ is quasi-minimal.

Here, a recursive lower bound is a recursive function $h:\N\rightarrow[0,1]$ such that $\lim_{n\rightarrow\infty}h(n)=1$ and for every $n$, and every $k\geq n$, $\frac{|{B\cap k}|}{k}\geq h(n)$.

\end{cor}

\begin{proof}

Let $B$ and $h$ be as in the statement of the corollary.

Let $A$ be a real, and assume that $B\geq_g \mathcal{R(A)}$.

Let $f$ be the function where $f(m)$ is the smallest number such that $\forall n> f(m),\ \frac{|{B\cap n}|}{n}>1-2^{-m}$.

Let $g$ be the function such that $g(m)$ is the smallest $n$ for which $h(n)>1-2^{-m}$.

Then, $g\gg f$, 
so by the proof of Proposition \ref{fastgrowing2}, $g$ computes $A$. Also, $g$ is recursive because $h$ is recursive, and $g$ is defined recursively in $h$. Therefore, $A$ is recursive, because it can be computed from a recursive function.

Thus, for any $A$, if $B\geq_g \mathcal{R(A)}$, then $A$ is recursive, and so $B$ is quasi-minimal.

\end{proof}

Before we move on to discuss questions of the structure of the generic degrees, we call attention to the asymmetry in our proof of Theorem~\ref{hyp}.

The proof of Proposition \ref{fastgrowing} shows that for any $f$, there is a density-1 real $B$, such that any generic oracle for $B$ can uniformly compute a function that dominates $f$. Thus, the entire content of $f$, as a modulus of computation, can be captured by a density-1 real, as a generic oracle.

We show here that the converse is not true, so density-1 generic degrees, in some sense, are more powerful than moduli of computation, but this cannot be seen from the Turing degrees that they can compute.

\begin{prop}\label{d1ismore}

There exists a density-1 real, $A$, such that for every $f:\N\rightarrow\N$, and every $\varphi$, there is a $g\gg f$ such that $\varphi^g$ is not a generic computation of $A$.

\end{prop}

To prove Proposition \ref{d1ismore}, we require a number of  technical lemmas concerning the specifics of how a reduction from a modulus of computation to a density-1 generic degree must work. The proof is somewhat long, in part because we will be required to somehow diagonalize against all possible moduli of computation --- an inherently uncountable set. We first introduce a modified version of more-is-more functionals that to works with oracles which are fast growing functions, rather than with partial oracles.

\begin{defn}\label{biggerisless}
\normalfont

Let $\varphi$ be a functional that outputs only 1's (as in Lemma \ref{output1}). Then $\varphi$ is a \emph{bigger-is-less} functional if for any functions $g,h$, if $h\gg g$, then $\text{dom}(h)\subseteq\text{dom}(g)$.

\end{defn}

\begin{lem}\label{biggerisless}

Let $A$ be a density-1 real, and let $f$ be a function such that there is a $\varphi$ so that for every $g\gg f$, $\varphi^g$ is a generic computation of $A$.

Then there exists a bigger-is-less functional $\psi$, which only outputs 1's (as in Lemma \ref{output1}) such that $\forall g$ if $g\gg f$, then $\psi^g$ is a generic computation of $A$.

\end{lem}

\begin{proof}

We first use Lemma \ref{output1} to replace $\varphi$ with a $\varphi$ that outputs only 1's when it halts.

To modify this $\varphi$ to a $\psi$ which is a bigger-is-less functional, we define $\psi$ so that $\psi^g(n)$ searches over all $h\gg g$ for some $h$ such that $\varphi^h(n)=1$, and if it finds such an $h$, then $\psi^g(n)=1$. Otherwise, $\psi^g(n)$ does not halt.

First, note that if there is an $h\gg g$, such that $\varphi^h(n)=1$, then eventually this will be learned, since only a finite amount of $h$ is used in the computation, and there are only countably many initial segments of functions that dominate $g$, and so $\psi$ can emulate $\varphi$ on all of these in parallel.

Second, note that if $h\gg g$, then $\text{dom}(h)\subseteq\text{dom}(g)$. This is because the set of functions that dominate $h$ is a subset of the set of functions which dominate $g$ (since $\gg$ is a transitive relationship.)

Finally, we show that for all $g\gg f$, $\psi^g$ is a generic computation of $A$. To show this, we must show two things. First we must show that $\text{dom}(\psi^g)$ is density-1, and second, we must show that $\text{dom}(\psi^g)\subseteq A$.

We know $\text{dom}(\psi^g)$ is density-1 because it contains $\text{dom}(\varphi^g)$ (because $g\gg g$, according to our definition of $\gg$.) Also, $\text{dom}(\psi^g)\subseteq A$ because, for any $h\gg g$, $h$ also dominates $f$. Then, $\varphi^h$ is a generic computation of $A$, and so, in particular, if $\varphi^h(n)=1$, then $n\in A$.

\end{proof}

\begin{defn}
\normalfont

Let $S$ be a real. Then we say that $S$ is \emph{sparse by design} if for every $n$, $|S\cap [5^n,5^{n+1}-1]|\leq2^{n+1}$.

\end{defn}

In other words, a sparse by design real is allowed to have at most 2 of the first 4 numbers, 4 of the next 16 numbers, 8 of the next 64 numbers, and so on.

\begin{lem}\label{sparseissparse}

If $\N\setminus A$ is sparse by design, then $A$ is density-1.

\end{lem}

\begin{proof}

The proof is a straightforward limit calculation.

If $5^m\leq n\leq 5^{m+1}$, then $\frac{|{A\cap n}|}{n}\geq \frac{n-2^{m+2}}{n}\geq\frac{5^m-2^{m+2}}{5^m}$. (This is because $A$ is missing at most $2+4+\dots+2^{m+1}=2^{m+2}-2$ many elements less than  $5^{m+1}$.) 

Since $\lim_{m\rightarrow\infty}\frac{5^m-2^{m+2}}{5^m}=1$, we may conclude that $\lim_{n\rightarrow\infty}\frac{|{A\cap n}|}{n}=1$.

\end{proof}

The basic idea of these sparse by design sets is that they will give us a tool for generating sufficiently many sufficiently general density-1 sets that we can diagonalize against all possible functions, but they are also sufficiently structured that we can cut and paste them together conveniently. The next lemma will give us the tool that we need to cut them into smaller sparse sets.

\begin{lem}\label{cutting}

Let $\varphi$ be a bigger-is-less functional. Let $S$ be a sparse set, and let $S=S_0\cup S_1$. Assume there exists an $f$ such that $\dom(\varphi^f)\cap S$ is empty. Let $n$ be minimal such that $f(0)=n$ for some such $f$.

\noindent Let $n_i$ be minimal such that $\exists f\, \, \dom(\varphi^{f_i})\cap S_i$ is empty.

\noindent Then, either $n_0=n$ or $n_1=n$.

\end{lem}

\begin{proof}

We first mention that $n_0$ and $n_1$ are both defined, because, if $\dom(\varphi^f)\cap S$ is empty, $\dom(\varphi^f)$ also does not intersect either $S_0$ or $S_1$. Thus, there exists an $f$ for each of those sets, and so there is a smallest value for $f(0)$. Also, since any $f$ that works for $S$ also works for $S_0$ and $S_1$, we have that $n_0\leq n$ and $n_1\leq n$.

We now prove the lemma by contradiction.

Assume $n_0<n$ and $n_1<n$.

For each $i$, fix $f_i$ such that $f_i(0)=n_i$, and $\dom(\varphi^{f_i})\cap S_i$ is empty.

Then, define $f$ so that, for each $m$, $f(m)=\max\lbrace f_0(m),f_1(m)\rbrace$.

Now, we claim that $\dom(\varphi^f)\cap S$ is empty.

To prove this, note that $\varphi$ is a bigger-is-less functional, so in particular, since $f\gg f_i$, we have that, for each $i$, $\dom(\varphi^f)\cap S_i$ is empty. So, because $S=S_0\cup S_1$, $\dom(\varphi^f)\cap S$ is empty.

Thus, there is an $f$ with $f(0)=\max\lbrace f_0(0),f_1(0)\rbrace<n$ such that $\dom(\varphi^f)\cap S$ is empty, contradicting our definition of $n$.

\end{proof}

We are now ready to prove Proposition \ref{d1ismore}, but to help clarify the construction, we first present the strategy for dealing with a single $\varphi$. Combining these constructions will be straightforward, since we are in no way confined to working uniformly.

\begin{lem}\label{d1ismore1}

Let $\varphi$ be a bigger-is-less functional. Then, either there is a sparse by design set $S$ such that for every $f$ $\dom(\varphi^f)\cap S$ is nonempty, or there is a $g$ so that for any $h\gg g$, $\dom(\varphi^h)$ is empty.

\end{lem}

\begin{proof}

Throughout this proof, the variable $S$ is assumed to only range over sets that are sparse by design.

Let $\varphi$ be a bigger-is-less functional, and assume that for every $S$, there is an $f$ so that $\dom(\varphi^f)\cap S$ is empty.
We will now build a $g$ such that for any $h\gg g$, $\dom(\varphi^h)$ is empty. We build $g$ by finite approximation.

To build $g$, we first prove the following statement.

\vspace{5pt}

\noindent\underline{Claim:} If $m\in\N$, and $\sigma$ is a finite partial function, with $\dom(\sigma)=m$,

\noindent and if $\forall S\ \exists f(f\upharpoonright m=\sigma)\wedge(\dom(\varphi^f)\cap S$ is empty),

\noindent then, there exists an $n$ such that $\forall S\ \exists f(f\upharpoonright m=\sigma)\wedge (f(m)=n)\wedge(\dom(\varphi^f)\cap S$ is empty).

\vspace{5pt}

\noindent\underline{Proof of Claim:} Assume $m\in\N$, $\dom(\sigma)=m$, $\forall S\ \exists f(f\upharpoonright m=\sigma)\wedge(\dom(\varphi^f)\cap S$ is empty).

For each sparse by design $S$, let $n_S$ be the smallest $n$ such that $\exists f(f\upharpoonright m=\sigma)\wedge (f(m)=n)\wedge(\dom(\varphi^f)\cap S$ is empty).

Then, we claim that $\lbrace n_S\rbrace$ is a bounded set.

If not, fix a sequence $S_i$ such that $n_{S_i}\geq i$.

Now, we will use Lemma \ref{cutting} to cut and paste the $S_i$ together into a new sparse by design set $T$ such that $n_T$ is undefined, which contradicts our assumption on $\sigma$.

To do this, for each $S_i$, let the ``first half" of $S_i$ be the union, over all $k$, of the first $2^k$ many elements of $S_i\cap [5^k,5^{k+1}-1]$, and let the second half of $S_i$ be the rest of the elements of $S_i$. Let $T_{i,0}$ be equal to either the first half, or the second half of $S_i$, whichever one will make $n_{T_{i,0}}=n_{S_i}$. (One of these will work, by Lemma \ref{cutting}.)

(It is possible that the second half is dramatically smaller than the first half, or even empty. Also, if, for some $k$, $S_i$ did not have $2^k$ elements of $[5^k,5^{k+1}-1]$, then just put all of those elements into $T_{i,0}$.)

The idea now is to continue to cut the $S_i$ in half, and then to paste together half of one, one fourth of another, one eighth of the next, and so on. Unfortunately, the $T_{i,0}$ each have at most one element between 1 and 4, and since we cannot cut this one element in half, we use the pigeonhole principle to get them to all agree on which first element to use.

So now, we put all of $T_{0,0}$ into $T$, and then, by the pigeonhole principle, fix $m_0\in [1,4]$ such that there are infinitely many $i$ such that either $m_0\in T_{i,0}$, or $T_{i,0}\cap [1,4]$ is empty. Then, remove the other $T_{i,0}$ from our list, and also remove $T_{0,0}$ from the list, and define $S_{i,0}$ to be the $i$th set that is left on the list.

We are now ready to cut each of our sets in half a second time, and just allow $m_0$ to be in both halves.

We now formalize this argument.

\vspace{5pt}

Fix $k$.

Assume that we have a sequence of sets $S_{i,k-1}, i\in\N$, such that for each $i$, $n_{S_{i,k-1}}\geq i+k$. Assume also that for $l<k$ we have numbers $m_l$ such that for each $i,l$, $S_{i,k-1}\cap[5^l,5^{l+1}-1]\subseteq\lbrace m_l\rbrace$. Also, assume that for $l\geq k$, $|S_{i,k-1}\cap [5^l,5^{l+1}-1]|\leq2^{l+1-k}$

Then, there exists a sequence of sets $S_{i,k}, i\in\N$, and a number $m_k$, such that the following things hold.
{\bf(1)} For each $i$, $n_{S_{i,k}}\geq (i+k+1)$.
{\bf(2)} For each $i$, and each $l\leq k$, $S_{i,k}\cap[5^l,5^{l+1}-1]\subseteq\lbrace m_l\rbrace$.
{\bf(3)} For each $l>k$, $|S_{i,k}\cap [5^l,5^{l+1}-1]|\leq2^{l-k}$.

The proof is effectively the same as for the case with $k=0$, and $S_{i,-1}=S_i$, but we present it here for completeness

For each $i$, let the ``stage $k$ first half" of $S_{i,k-1}$ be the union, over all $l\geq k$, of the first $2^{l-k}$ many elements of $S_{i,k-1}\cap [5^l,5^{l+1}-1]$, and let the second half of $S_i$ be the rest of the elements of $S_i$. Let $T_{i,k}$ be equal to either the stage $k$ first half, or the stage $k$ second half of $S_i$, whichever one will make $n_{T_{i,k}}=n_{S_{i,k-1}}$.

(The ``first halves" of the sets as defined here will not have any elements less than $5^k$, but that is not a problem for us.)

Note now that for each $l>k,i$, $|T_{i,k}\cap [5^l,5^{l+1}-1]|\leq2^{l-k}$.

Then, by the pigeonhole principle, fix $m_k\in [5^k,5^{k+1}-1]$ such that there are infinitely many $i$ such that either $m_k\in T_{i,k}$, or $T_{i,k}\cap [5^k,5^{k+1}-1]$ is empty. Then, remove the other $T_{i,k}$ from our list, and also remove $T_{0,k}$ from the list, and define $S_{i,k}$ to be the $i$th set that is left on the list.

Then, we may conclude the following.

\noindent{\bf(1)} For each $i$, $n_{S_{i,k}}\geq i+k+1$ because $S_{i,k}=T_{j,k}$ for some $j>i$, and $n_{T_{j,k}}=n_{S_{j,k-1}}$ by construction.

\noindent{\bf(2)} For each $i,l\leq k$, $S_{i,k}\cap[5^l,5^{l+1}-1]\subseteq\lbrace m_l\rbrace$ because that was how we chose the $S_{i,k}$ as a subsequence of the $T_{i,k}$.

\noindent{\bf(3)} For each $l>k$, $|S_{i,k}\cap [5^l,5^{l+1}-1]|\leq2^{l-k}$ because of the way that we cut the $S_{i,k-1}$ in half when building the $T_{i,k}$.

This concludes the proof of our claim.

\vspace{5pt}

Now, we construct $T$ by simply letting $T$ be the union, over all $k$, of the $T_{0,k}$ that we used in our construction. We claim two things. First, $T$ is a sparse by design set. Second, $n_T$ is undefined.

To show that $T$ is sparse by design, we show that for each $l$, $|T\cap[5^l,5^{l+1}-1]|\leq2^{l+1}$. This is because, for each $k\leq l$, $|T_{0,k}\cap [5^l,5^{l+1}-1]|\leq2^{l-k}$. Also, for each $k>l$, $T_{0,k}\cap [5^l,5^{l+1}-1]\subseteq\lbrace m_l\rbrace$. Thus, $|T\cap [5^l,5^{l+1}-1]|\leq 1+1+2+\dots +2^l=2^{l+1}$.

To show that $n_T$ is undefined, note that for each $k\in N$, $T_{0,k}\subseteq T$ and that $n_{T_{0,k}}=n_{S_{0,k}}\geq k+1$. Then, $n_T\geq n_{T_{0,k}}$ because any $f$ that does not halt on any of $T$ also does not halt on $T_{0,k}$. Thus, for every $k\in\N$, $n_T>k$. So $n_T$ cannot be any natural number, and thus is undefined.

This provides us the contradiction that we need in order to prove that $\lbrace n_S\rbrace$ is a bounded set.

\vspace{5pt}

Thus, we may conclude that if $m\in\N$, and $\sigma$ is a finite partial function, with $\dom(\sigma)=m$, and

\noindent if $\forall S\ \exists f(f\upharpoonright m=\sigma)\wedge(\dom(\varphi^f)\cap S$ is empty),

\noindent then, there exists an $n$ such that $\forall S\ \exists f(f\upharpoonright m=\sigma)\wedge (f(m)=n)\wedge(\dom(\varphi^f)\cap S$ is empty).

\vspace{5pt}

We are now ready to define the function $g$ by induction.

Start with $\sigma$ being the empty string, and at each stage $s$, let $\sigma$ be the portion of $g$ that has been constructed so far, and let $g(s)$ be the $n$ that is given to us by this claim.

We claim that $\dom(\varphi^g)$ is empty.

To show this, let $n>0$, and assume $n\in\dom(\varphi^g)$. Let $S=\lbrace n\rbrace$. It is clear that $S$ is sparse by design. Let $\sigma$ be the initial segment of $g$ including all of the numbers that were queried by $\varphi$ during the computation of $\varphi^g(n)$. Let $m=\dom(\sigma)$. By construction of $g$, $\exists f(f\upharpoonright m=\sigma)\wedge(\dom(\varphi^f)\cap S$ is empty), but this is a contradiction, because for any $f$, if $f\upharpoonright m=\sigma$ then $\varphi^f(n)\downarrow$, because it is equal to $g$ on the part of $g$ that was queried to make $\varphi^g(n)\downarrow$.

Thus, $\dom(\varphi^g)$ is empty, and we have proved the lemma.

\end{proof}

We are now ready to prove Proposition \ref{d1ismore}.

The basic idea of the proof is to build $A$ by splitting $\N$ into infinitely many pieces, each of positive density. Each piece will be used to diagonalize against a specific $\varphi$ using Lemma \ref{d1ismore1}.

\begin{proof}

For each $e$, we will build a real $A_e$, and then we will let $A$ be the real given by $(2n+1)2^e\in A\longleftrightarrow n\in A_e$. Note that $A$ will be density-1 if and only if each $A_e$ is density-1.

(If $A_e$ is not density-1, then there is some $\epsilon>0$ so that the liminf of its initial segment densities is $1-\epsilon$. Then the liminf of the initial segment densities of $A$ must be $\leq1-2^{-e}\epsilon$. The converse is similar.)

Now, for each $e$, we will construct $A_e$ to ensure that there is no $f$ such that $\forall g\gg f\ \varphi^g_e$ is a generic computation of $A$. By Lemma \ref{biggerisless}, we need only concern ourselves with bigger-is-less functionals, so if $\varphi_e$ is not a bigger-is-less functional, then we simply let $A_e=\N$. 

If $\varphi_e$ is a bigger-is-less functional, then let $\psi_e$ be the functional where $\psi_e^X(n)=\varphi_e^X((2n+1)2^e)$. Note then that if $\varphi_e^g$ is a generic computation of $A$, then $\psi_e^g$ is a generic computation of $A_e$. Note also that $\psi_e$ is also a bigger-is-less functional.

Thus, we may use Lemma \ref{d1ismore1} to show that either there is a sparse by design set $S$ such that for every $f$ $\dom(\psi_e^f)\cap S$ is nonempty, or there is a $g$ so that for any $h\gg g$, $\dom(\psi_e^h)$ is empty.

In the first case, let $A_e$ be $N\setminus S_e$, where $S_e$ is the set $S$ that is given to us by the lemma. Then for every $f$ $\dom(\psi_e^f)\cap S_e$ is nonempty, and so, for any $f$, $\psi_e^f$ is not a generic computation of $A_e$, since it halts somewhere outside of $A_e$, and we require that bigger-is-less functionals only output 1's when they halt. Thus, we have ensured that, for any $f$, $\varphi_e^f$ is not a generic computation of $A$.

In the second case, let $A_e=\N$, and let $g_e$ be the function $g$ that is given to us by the lemma. For any $f$, $\psi_e^f$ is not a generic computation of $A_e$, because $\sup(f,g_e)\gg g_e$. Thus, by the lemma, $\dom(\psi_e^{\sup(f,g_e)})$ is empty, and thus $\psi_e^{\sup(f,g_e)}$ is not a generic computation of $A_e$. Also, $\sup(f,g_e)\gg f$, so there is a $g\gg f$ such that $\psi_e^g$ is not a generic computation of $A_e$. Therefore, $\varphi_e^g$ is also not a generic computation of $A$.

Thus, we have built an $A$ where, for every bigger-is-less $\varphi$, for every $f:\N\rightarrow\N$, there is a $g\gg f$ such that $\varphi^g$ is not a generic computation of $A$, and so, by Lemma \ref{biggerisless}, for every $\varphi$, and every $f$, there is a $g\gg f$ such that $\varphi^g$ is not a generic computation of $A$.

\end{proof}

We have shown that there are density-1 generic degrees whose generic information cannot be recovered from a sufficiently fast growing function. Thus, density-1 generic degrees can, in some sense, be thought of as a generalization of fast growing functions.

We now move on to discuss the structure of the generic degrees as a whole, and to see how questions about the density-1 degrees fit in to the larger picture.

\subsection{Structure of the generic degrees}

The main purpose of this section is to address the following two questions about generic degrees.

\begin{question}\label{q1}
Do there exist minimal degrees in the generic degrees?
\end{question}

In other words, is there a generic degree {\bf a} such that ${\bf a}>_g0$, and for all generic degrees {\bf b}, if ${\bf a}\geq_g{\bf b}>_g0$, then ${\bf a}={\bf b}$?

\begin{question}\label{q2}
Do there exist minimal pairs in the generic degrees?
\end{question}

In other words, do there exist generic degrees {\bf a} and {\bf a} such that ${\bf a}>_g0$, ${\bf b}>_g0$, and for all generic degrees {\bf c}, if ${\bf a}\geq_g{\bf c}$, and ${\bf a}\geq_g{\bf c}$, then ${\bf c}=0$?

To aid in our discussion, we also present an open question concerning the degrees of the density-1 sets, which will turn out to be relevant to Questions \ref{q1} and \ref{q2}.

\begin{question}\label{q3}

Is it true that for every nonzero generic degree {\textbf b} there exists a nonzero density-1 generic degree {\textbf a} such that ${\textbf b}\geq_g{\textbf a}$?

\end{question}

We will see that a ``yes" answer to Question \ref{q3} would imply that the answer to Question \ref{q1} is no, and that a ``no" answer to Question \ref{q3} would imply that the answer to Question \ref{q2} is yes.

Before we do this, we give a brief overview on what is known about the degrees that have a density-1 degree below them. In a previous paper \cite{I}, we showed that there are no minimal pairs for generic computation (not generic reduction.) Translated into the language of generic reduction, this can be phrased as follows.

\begin{thm}\label{nominimalpair}
\normalfont{(I.)} \cite{I}

\emph{Let $A, B$ be nonrecursive reals. Then $\mathcal{R}(A)$ and $\mathcal{R}(B)$ do not form a minimal pair in the generic degrees.} 

\end{thm}

An analysis of the proof, however, yields a stronger result that is also more relevant to the current discussion.

\begin{prop}\label{nominimalpair1}
\normalfont{(I.)} \cite{I}

\emph{Let $A, B$ be nonrecursive reals. Then there is a $C$ such that $C$ is density-1, $C$ is not generically computable, $\mathcal{R}(A)\geq_gC$, and $\mathcal{R}(B)\geq_g~C$.}

\end{prop}

In fact, an easy generalization of the argument from \cite{I} can be used to strengthen Proposition \ref{nominimalpair1} to work for any finite set of nonrecursive reals.

Proposition \ref{nominimalpair1} gives us an immediate corollary concerning the quasi-minimal generic degrees.

\begin{cor}
Let {\textbf b} be any non-quasi-minimal generic degree. Then there is a density-1 quasi-minimal generic degree {\textbf a} such that ${\textbf b}\geq_g{\textbf a}$.

\end{cor}

This, in particular, shows that a counterexample to Question \ref{q3} would necessarily be a quasi-minimal generic degree.

\begin{proof}
Let {\textbf b} be a non-quasi-minimal generic degree. Let {\textbf c} be a nonzero Turing degree such that ${\textbf b}\geq_g\mathcal{R}({\textbf c})$. Let $C\in{\textbf c}$.

Then, by a result of classical recursion theory, $C$ is half a minimal pair, so there is a nonrecursive real $D$ such that $C$ and $D$ form a minimal pair in the Turing degrees. (For any $X$, if $C\geq_TX$ and $D\geq_TX$, then $X$ is recursive.)

By Proposition \ref{nominimalpair1}, we can choose a density-1 real $A$ such that $A$ is not generically computable, $\mathcal{R}(C)\geq_gA$, and $\mathcal{R}(D)\geq_gA$. Let {\textbf a} be the generic degree of $A$. Then, ${\textbf b}\geq {\textbf c}\geq {\textbf a}$, and also {\textbf a} is density-1. It remains to show that {\textbf a} is quasi-minimal.

If not, then fix $X$ nonrecursive such that $\mathcal{R}(X)\leq_gA$. but then, $\mathcal{R}(C)\geq_gA\geq\mathcal{R}(X)$, and also $\mathcal{R}(D)\geq_gA\geq\mathcal{R}(X)$. Since $\mathcal{R}$ induces an embedding of the Turing degrees into the generic degrees, we get that $C\geq_TX$ and $D\geq_TX$. This contradicts that $C$ and $D$ form a minimal pair in the Turing degrees, and so $X$ could not have existed, and so {\textbf a} is quasi-minimal.

\end{proof}

Thus, there are a fair number of quasi-minimal generic degrees, and, in fact, many of them are density-1. Next, we indicate how a resolution to Question \ref{q3} would allow us to resolve either Question \ref{q1} or Question \ref{q2}.

Showing that a positive resolution to Question \ref{q3} gives a negative resolution to Question \ref{q2} is the easy half, since it is a direct application of Proposition \ref{denseisdense}.

\begin{prop}\label{a1}
If for every nonzero generic degree {\textbf b} there exists a nonzero density-1 generic degree {\textbf a} such that ${\textbf a}\leq_g{\textbf b}$, then there are no minimal degrees in the generic degrees.
\end{prop}

\begin{proof}
Let {\textbf b} be a generic degree, and assume that there exists a nonzero density-1 generic degree {\textbf a} such that ${\textbf a}\leq_g{\textbf b}$.

Then, 0 is a density-1 generic degree (because $\mathbb{N}$ is generically computable, and density-1). So by Proposition \ref{denseisdense}, there exists a density-1 generic degree {\textbf c}, such that ${\textbf a}>_g{\textbf c}>_g0$. Thus, ${\textbf b}>_g{\textbf c}>_g0$, and so {\textbf b} is not a minimal generic degree.

\end{proof}

The other half is slightly more subtle, but the argument is basically a modification of the usual construction of a minimal pair in the Turing degrees, together with the realization that a counterexample to Question \ref{q3} would have exactly the property that we require in order to adapt the construction to our situation.

\begin{prop}\label{a2}
If there exists a nonzero generic degree {\textbf b} such that there is no nonzero density-1 generic degree {\textbf a} with ${\textbf a}\leq_g{\textbf b}$, then {\textbf b} is half of a minimal pair in the generic degrees.

\end{prop}

\begin{proof}

Let $B\in {\textbf b}$. We will build a real $C$ such that $\mathcal{R}(C)$ and $B$ form a minimal pair for generic reduction, or in other words, so that if $B\geq_gD$, and $C$ generically computes $D$, then $D$ is generically computable.

We build $C$ by finite approximation.

We have one stage for each $e$, and one stage for each $\langle i,j\rangle$.

At stage $e$, we ensure that $C$ is not computed by $\varphi_e$ in the usual manner. (We have a $\tau$ which is our current approximation to $C$. We ask whether there exists an $n>|\sigma|$ such that $\varphi_e(n)\downarrow$, and if there does, then we extend $\tau$ so that $\tau\neq \varphi_e(n)$. If there does not exist such an $n$, then in particular $\varphi_e$ is not total, so we do not need to to anything to ensure that $C$ is not computed by $\varphi_e$.)

At stage $\langle i,j\rangle$, we have an approximation $\tau$ for $C$, and we need to ensure that if $D$ generically reduces to $B$ via $\varphi_i$, and if $D$ is generically computable from $C$ via $\varphi_j$, then $D$ is generically computable.

(As a reminder, being generically computable from $C$ is equivalent to being generically reducible to $\mathcal{R}(C)$. Also, being generically equivalent to 0 equivalent to being generically computable.)

The first thing that we ask is whether there is any extension $\tilde\tau$ of $\tau$, and any generic oracle $(B)$ for $B$ such that for some $n$, $\varphi_i^{(B)}(n)\neq\varphi_j^{\tilde\tau}(n)$. If there is, then we extend $\tau$ to $\tilde\tau$, thereby ensuring that $\varphi_j^C$ cannot be a generic computation of any real that generically reduces to $B$ via $\varphi_i$. (This is because neither computation is allowed to make any mistakes, and so in particular, if $\varphi_i^{(B)}$, and $\varphi_j^{\tilde\tau}$ are trying to be generic computations of the same real, then they are not allowed to disagree with each other anywhere.)

If there is no such $\tilde\tau$, then we ask whether there is any $D$, such that $D$ generically reduces to $B$ via $\varphi_i$. (In other words, we ask whether it is true that for every generic oracle $(B)$, for $A$, $\text{dom}\big(\varphi_i^{(B)}\big)$ is density-1, and whether it is true that the $\varphi_i^{(B)}$ all agree wherever they halt. If either of these is false, then there cannot be any $D$ that generically reduces to $B$ via $\varphi_i$.) If there is no such $D$, then, again, we do not need to do anything at stage $\langle i,j\rangle$.

If there is no such $\tilde\tau$, and if there does exist such a $D$, then we may generically compute any $D$ that generically reduces to $B$ via $\varphi_i$, and that also is generically computable from $C$ via $\varphi_j$ by the following algorithm. The basic idea of the algorithm will be that it halts on a subset of the halting set of $\varphi_i$, and gives only the outputs given by $\varphi_j$ applied to extensions of $\tau$.

The key thing to notice here is that the union, over all generic oracles $(B)$, for $B$, of the domains of $\varphi_i^{(B)}$ is a density-1 set that generically reduces to $B$. This is witnessed by the algorithm, $\psi$, that halts wherever $\varphi_i$ halts, and that outputs a 1 wherever it halts. So $\psi^X(n)\downarrow\Longleftrightarrow \varphi^X(n)\downarrow\Longleftrightarrow \psi^X(n)=1$. (Let $A$ be the union, over all generic oracles, $(B)$, for $B$, of the domains of the $\varphi_i^{(B)}$. Then, for each $(B)$, $\text{dom}\big(\varphi_i^{(B)}\big)$ is a density-1 set. Also, it is clearly a subset of $A$, and so $\psi^{(B)}$ is a generic computation of $A$.)

Thus, by the hypothesis on $B$, $A$ is generically computable (since it is a density-1 set, and $B\geq_gA$), and so $A$ contains a density-1 r.e. subset, $W$. (By Lemma \ref{output1}, a density-1 set is generically computable if and only if it contains a density-1 r.e. subset, since a generic computation of $A$ that only outputs 1's actually is the same thing as an enumeration of a density-1 subset of $A$.)

Now, we define $\psi$ so that $\psi(n)\downarrow$ if and only if $n\in W$ and there exists a $\tilde\tau$ extending $\tau$ such that $\varphi_j^{\tilde\tau}(n)\downarrow$. In this case, we let $\psi(n)=\varphi_j^{\tilde\tau}(n)$ for the first such $\tilde\tau$ that we find.

This is a generic computation of any $D$ that we are concerned with, because the halting set is contained in the union of the halting sets of the $\varphi_i^{(B)}$, and the fact that we couldn't diagonalize means that the output that we found from our $\tilde\tau$ using $\varphi_j$ must be the same as that given by $\varphi_i^{(B)}$. This value must be correct, since we are assuming that $D$ reduces to $B$ via $\varphi_i$.

\end{proof}

Combining Propositions \ref{a1} and \ref{a2} gives us a free corollary concerning the generic degrees.

\begin{cor}\label{minismin}

If there exist minimal generic degrees, then there exist minimal pairs of generic degrees. In fact, any minimal generic degree is half of a minimal pair in the generic degrees.

\end{cor}

This is simply because the answer to Question \ref{q3} must either be ``yes," or ``no."

This seems trivial, since the only way for Corollary \ref{minismin} to be false would be if there were a single minimal generic degree that was below all the other nontrivial generic degrees. In the Turing degrees, it is easy to show that any nontrivial Turing degree has another Turing degree that is incomparable to it, but we do not have a proof that this is true in the generic degrees. It seems highly unlikely for the generic degrees to have an ``hourglass" shape, with a non-minimal generic degree that is comparable to all others, but we do not currently have a method of ruling this out.

\begin{question}

Does there exist a generic degree {\textbf a} such that for every generic degree {\textbf b}, either ${\textbf a}\geq_g{\textbf b}$, or ${\textbf b}\geq_g{\textbf a}$?

\end{question}

Note that such an {\textbf a} would necessarily be quasi-minimal.

\begin{obs}

If there is a generic degree that is comparable to all other generic degrees, then that generic degree is quasi-minimal.

\end{obs}

\begin{proof}

Let {\textbf a} be a generic degree that is comparable to all other generic degrees. Let $A\in{\textbf a}$, and let $A\geq_g \mathcal{R}(B)$. We must show that $B$ is recursive.

Note that if $A\geq_g \mathcal{R}(B)$, then $A\geq_T B$. (It is easy for $A$ to compute a generic oracle for itself, and any generic oracle for $A$ can generically compute $\mathcal{R}(B)$, and so can compute $B$.) If $B$ is nonrecursive, then we can choose $C$ to be Turing incomparable to both $A$, and $B$. Then $A\ngeq_g \mathcal{R}(C)$, since $A\ngeq_TC$, and $A\nleq_g \mathcal{R}(C)$, since $C\ngeq_TB$. Therefore, we cannot have $\mathcal{R}(C)\geq_gA\geq_g\mathcal{R}(B)$.

Thus, if $A\geq_g \mathcal{R}(B)$, then $B$ must be recursive, and so, {\textbf a} must be quasi-minimal.

\end{proof}

\section{${\bf\Pi}^1_1$-completeness}

We finish by proving that generic reduction is ${\bf\Pi}^1_1$-complete. 

The definition we use for generic reduction is intrinsically ${\bf\Pi}^1_1$, since it involves a universal quantifier over all generic oracles. This makes it very difficult to run the constructions that we want to, since the techniques of recursion theory are often poorly suited to dealing with quantifiers over uncountable sets. Indeed, this is one of the primary reasons that most of our work concerns the generic degrees of density-1 sets --- it is somewhat easier to work with density-1 subsets of a given set than with partial oracles with density-1 domain.

By showing that $\geq_g$ is $\mathbf{\Pi_1^1}$-complete, we show that there is no way for this quantifier over generic oracles to be replaced by quantifiers over naturals, and that generic reduction is, in some sense, as complicated as it seems.

We show that $\geq_g$ is $\mathbf{\Pi_1^1}$-complete by showing that from any tree $T\subseteq\omega^\omega$, we can use the jump of that tree to uniformly find $A$ and $B$ such that $A\geq_gB$ if and only if $T$ is well-founded. This gives a Borel reduction of well-foundedness to generic reducibility, proving that generic reducibility is $\mathbf{\Pi_1^1}$-complete.

\begin{thm}\label{pi11}

There exists an algorithm which, given a tree $T\subseteq\omega^\omega$, uses $T^\prime$ as an oracle to compute a pair of reals $A$ and $B$ such that $A\geq_gB$ if and only if $T$ is well-founded. Thus, $\geq_g$ is $\mathbf{\Pi_1^1}$-complete.

\end{thm}

The proof will consist of three parts.

In the first part, we describe the intended reduction from $A$ to $B$. The reduction will have the property that an infinite path through $T$ corresponds to a method of creating a generic oracle for $A$ that does not generically compute $B$ via the intended reduction. Every node of the path will be able to be translated into another drop in the density of the domain of the computation without a corresponding drop in the density of the domain of the oracle.

In the second part, we build $A$ and $B$ to ensure that no reduction other than the intended reduction will work. During this second part, we do not need to work effectively, but rather we have access to $T^\prime$.

Finally, we verify that our construction works. If $T$ is well-founded, this will be clear, because the intended reduction will function as it is designed to. If $T$ is ill-founded, we will need to show that for any potential generic reduction $\varphi$, there is a generic oracle for $A$ that either makes $\varphi$ give a false answer somewhere, or makes dom$(\varphi)$ not be density-1. During this third part, we are not forced to work effectively in anything, and indeed, the generic oracles that we build would be quite difficult to compute.

The proof is somewhat bogged down in notation, but during the proof, we will attempt to explain the notation as it appears.

\begin{proof}
\textbf {Part 1}
\vspace{5pt}

Let $T\subseteq\omega^\omega$.

For each $\sigma\in T$, there will be a single bit $b_\sigma\in\lbrace0,1\rbrace$.

We will code $b_\sigma$ into $A$ in a manner so that any partial oracle for $A$ which cannot recover $b_\sigma$ must have its density drop below $1-2^{-|\sigma|-2}$ at least once as a result.

We will code $b_\sigma$ into $B$ in a manner so that if a computation cannot compute $b_\sigma$, then the domain of that computation will have its density drop below $\frac12$ as a result.

The intention of this is that if there is an infinite path $Q$ through $T$, we will be able to produce a generic oracle for $A$ that omits $b_\sigma$ for every $\sigma$ on $P$, and that therefore cannot generically compute $B$, because it is missing infinitely many pieces of information, and each missing piece of information forces there to be another instance of the computation's domain's density dropping below $\frac12$.

Unfortunately, this creates a problem. There could theoretically be a generic oracle for $A$ that chooses a collection of $\sigma$'s of increasing lengths from \emph{different} paths of $T$, and omits each of the corresponding $b_\sigma$'s. Potentially this might be unable to generically compute $B$ even if the tree is well-founded.

For this reason, we need to also introduce a method for propagating information down the tree: if $\sigma\prec\tau$, and $b_\sigma$ is known, then it should be easy to deduce what $b_\tau$ is. That way, removing the knowledge of an entire branch will still have the original desired effect, but removing bits of information from different branches will be much more difficult than previously.

However, if we want to be able to remove information along a path, we need to make sure that our procedure for propagating information downward along $T$ does not also cause information to propagate upward along $T$. Else, if $Q$ is a path through $T$, $\sigma\prec\tau$, $\sigma\prec Q$, and $\tau\nprec Q$, then $b_\sigma$ could be deduced from $b_\tau$, so we would not be able to selectively remove only the information along $Q$ from $A$.

For each $\sigma\in T$, for each $m<|\sigma|$, we create a procedure to deduce $b_\sigma$ from $\langle b_{\sigma_0},...,b_{\sigma_m}\rangle$, where $\sigma_{-1}=\sigma$, and $\sigma_{i+1}$ is the immediate predecessor of $\sigma_i$. This procedure is coded in a way so that if a partial oracle for $A$ does not know the procedure, then its density must drop below $1-2^{|\sigma_m|+2}$ as a result. The procedure is also coded in a way so that knowing $b_\sigma$ and knowing the procedure does not necessarily allow us to deduce any of the $b_{\sigma_i}$.

The actual coding is as follows:
\vspace{5pt}

Consider the sets $P_i=\lbrace n\in\N\, |\, 2^i\leq n<2^{i+1}\rbrace$.

In $B$, for each $\sigma\in\omega^\omega$, uniformly choose an $i$, and code $b_\sigma$ into $P_i$. (If $n\in P_i$, then $n\in B\leftrightarrow b_\sigma=1$.) If $\sigma\notin T$, then $b_\sigma=0$.

This will ensure that a generic computation of $B$ must compute all but finitely many of the $b_\sigma$, since omitting a finite number of $P_i$ is only a finite omission, but no density-1 set can avoid infinitely many $P_i$.

We define $A$ to be equal to $\widetilde A\oplus\mathcal{R}(T)$, where $\widetilde A$ is built as follows.

We will use some of the $P_i$ to code the values of the $b_\sigma$ in $\widetilde A$, and use the rest to code deduction procedures in $\widetilde A$.

For each $\sigma\in\omega^\omega$, uniformly choose an (even) $i$, and code $b_\sigma$ into the last $\frac1{2^{|\sigma|}}$ of $P_i$. (If $n$ is one of the last $2^{i-|\sigma|}$ many elements of $P_i$, then $n\in\widetilde A\leftrightarrow b_\sigma=1$. If $n$ is a smaller element of $P_i$ then $n\notin\widetilde A$.)

Uniformly choose a $P_i$ (with $i$ odd) for each sequence $\langle \sigma, m, \tau, j, k\rangle$ such that $\sigma\in\omega^\omega,m<|\sigma|, \tau\in2^m, j\in\lbrace0,1\rbrace, k\in\mathbb{N}$. Call it $P_{\sigma,m,\tau,j,k}$. Then, to deduce $b_\sigma$ from the sequence $\langle b_{\sigma_0},...,b_{\sigma_m}\rangle$, we use the following formula.
$$b_\sigma=j\Longleftrightarrow\exists n\exists k\  n\in \widetilde A\cap P_{\sigma,m,\langle b_{\sigma_0},...,b_{\sigma_m}\rangle,j,k}$$
When we actually build $\widetilde A$, we will ensure that for exactly one value of $k$, we put the last $\frac1{2^{|\sigma|-m-1}}\left|P_{\sigma,m,\langle b_{\sigma_0},...,b_{\sigma_m}\rangle,j,k}\right|$ many elements of $P_{\sigma,m,\langle b_{\sigma_0},...,b_{\sigma_m}\rangle,j,k}$ into $\widetilde A$.

For a fixed value of $\langle \sigma, m, \tau\rangle$, the set of all $P_{\sigma,m,\tau,j,k}$ is known as the \emph{deduction procedure} coding $b_\sigma$ from its $m+1$ predecessors. The deduction procedure \emph{operates under true assumptions} if $\tau = \langle b_{\sigma_0},...,b_{\sigma_m}\rangle$, and it \emph{operates under false assumptions} if $\tau \neq \langle b_{\sigma_0},...,b_{\sigma_m}\rangle$.

The idea here is that knowing $\langle b_{\sigma_0},...,b_{\sigma_m}\rangle$ will direct you to the correct place to look for the value of $b_\sigma$. Once you know where to look, you simply search until you find an answer. If you try to search for the value of $b_\sigma$ using incorrect values for $\langle b_{\sigma_0},...,b_{\sigma_m}\rangle$, then you might get the correct answer, you might get the incorrect answer, and you might get no answer. Because of this, knowing $b_\sigma$ gives little to no information about $\langle b_{\sigma_0},...,b_{\sigma_m}\rangle$. However, if we wish to remove a deduction procedure from an oracle, we only need to remove the place where it actually gives an answer, i.e., the last $\frac{2^i}{2^{|\sigma|-m-1}}$ many elements of $P_{i}$ for some $i$. The size is calibrated so that removing a deduction procedure whose shortest element is $\tau$ is just as difficult as removing the knowledge of what $b_\tau$ is.

Given that $A$ and $B$ are each built in the manner just described, a generic oracle for $A$ is able to generically compute $B$ by the following algorithm.

\vspace{5pt}

Let $(A)$ be a generic oracle for $A$.

To determine whether or not $n\in B$, we first determine which $P_i$ $n$ is in. Then we determine which $b_\sigma$ is coded into that $P_i$. Then, since $A\geq_g\mathcal{R}(T)$, we can use $(A)$ to determine whether or not $\sigma\in T$. If no, then $n\notin B$. If yes, then we attempt to determine the value of $b_\sigma$ as follows.

We define the sentence ``$(A)$ can determine the value of $b_\sigma$." by induction on $|\sigma|$.

Let $P_i$ be the set assigned to code $b_\sigma$ in to $\widetilde A$. Then, if $(A)$ gives an output on one of the last $2^{i-|\sigma|}$ many elements of $P_i$, then $(A)$ can determine the value of $b_\sigma$, and that value is the value of the output that we found.

The other way that $(A)$ can determine the value of $b_\sigma$ is with our deduction procedures. If there is some $m<|\sigma|$ such that $(A)$ can determine the values of $b_\tau$ for $\tau$ equal to one of the $m+1$ immediate predecessors of $\sigma$, and if $(A)$ also includes the value of $A$ in the location where the corresponding deduction procedure has a $1$, then the deduction procedure allows $(A)$ to determine the value of $b_\sigma$ just as when we described the deduction procedures.

\vspace{5pt}

If $T$ is well-founded, then for any generic oracle $(A)$, for $A$, there will only be finitely many $\sigma$ such that $(A)$ cannot determine the value of $b_\sigma$. The proof is as follows.

Let $(A)$ be a generic oracle for $A$. Assume there are infinitely many $\sigma$ such that $(A)$ cannot determine the value of $b_\sigma$. Let $\widetilde T\subseteq\omega^\omega$ be the smallest subtree of $T$ containing all of the $\sigma$ such that $(A)$ cannot determine the value of $b_\sigma$. Then, $\widetilde T$ is well-founded, because it is contained in $T$, which is well-founded. Also, it is infinite, by assumption. Thus, it must have at least one node where it branches infinitely. Call the first such node $\sigma_0$.

From each of those countably many branches, choose a minimal node $\sigma$ such that the generic oracle cannot determine the value of $b_\sigma$. (For $i>0$, let $\sigma_i$ be an extension of the $i$th branch of $\widetilde T$ such that $(A)$ cannot determine the value of $b_{\sigma_i}$, but for any $\tau$, if $\sigma_0\prec\tau\prec\sigma_i$, then $(A)$ can determine the value of $b_\tau$.)

Then, we claim that the domain of $(A)$ must have its density drop below $1-2^{-|\sigma_0|-3}$ infinitely often.

The reason for this is that, for any given $i$, if $\sigma_i$ is an immediate successor to $\sigma_0$, then the generic oracle must have its density drop below $1-2^{-|\sigma_0|-3}$ to not know the value of $b_{\sigma_i}$. (This is by the manner in which the $b_{\sigma_i}$ are directly coded into $A$.) Otherwise, by assumption on $\sigma_i$, we know that the generic oracle \emph{can} determine the values of the $b_\tau$ for all of the $\tau$ that satisfy $\sigma_0\prec\tau\prec\sigma_i$. Thus, since the first of those $\tau$ has length $|\sigma_0|+1$, the deduction procedure that allows us to determine $b_{\sigma_i}$ from those $b_\tau$ is coded in a way so that if the generic oracle cannot recover that deduction procedure, then its density must drop below $1-2^{-|\sigma_0|-3}$.

Each of these things is coded in a different place, so the domain of the generic oracle has its density drop below $1-2^{-|\sigma_0|-3}$ infinitely often. Therefore, the domain is not density-1, so $(A)$ is not a generic oracle, providing a contradiction.

Thus, if $T$ is well-founded, then any generic oracle can (uniformly) recover all but finitely many of the $b_\sigma$, and therefore all but finitely many of the bits of $B$. Thus, $A\geq_gB$.

\vspace{5pt}

\textbf{Part 2}

\vspace{5pt}

In this part, we construct $\widetilde{A}$ (and therefore $A$, and $B$, where these are defined according to the construction outlined in Part 1). While we do this, we ensure that if $T$ is ill-founded, then $A\ngeq_gB$.

If $T$ is ill-founded, then the intended reduction will not work as a generic reduction. Thus, the main purpose of this section will be ensuring that any reduction that ``cheats" infinitely often occasionally makes mistakes, and therefore cannot be used to generically reduce $B$ to $A$, since a generic reduction is never allowed to give incorrect outputs. (Here, ``cheating" simply means guessing at the values of $b_\sigma$ without having solid evidence as to why those guesses should be correct.)

For those who like to think about constructions in terms of forcing, we will be building a generic with respect to the poset implicitly defined in Part 1, except with the additional caveats that on the ``incorrect" deduction procedures, we are only allowed to encode them giving one output, giving the other output, or giving no output. Furthermore, our conditions are allowed to include restrictions of the form in the following paragraph.

One of our key techniques will be to fix a finite set of numbers, and demand that for any $n$ in that set, any deduction procedure that operates under false assumptions about any $b_\sigma$ with $|\sigma|=n$ does not produce any answers. (More formally, for any $n_0$ in that set, and any $\sigma_0, \sigma_1$, with $\sigma_0\prec\sigma_1$, and $|\sigma_0|=n_0$, $|\sigma_1|=n_1>n_0$, for any $m\geq n_1-n_0-1$, for any $\tau$ such that $\tau(n_1-n_0-1)\neq b_{\sigma_0}$, for any j, and k, $P_{\sigma_1,m,\tau,j,k}\cap\widetilde A$ is empty.)

\vspace{5pt}

The actual construction is as follows.

At the beginning of stage $s$, there is some number $f(s)$ such that we have determined the values of $b_\sigma$ for every $\sigma$ such that $|\sigma|<f(s)$, and for no other $\sigma$. For every $\sigma$ with $|\sigma|<f(s)$, and every deduction procedure for computing $b_\sigma$, we have determined the values of $\widetilde A$ on the entire deduction procedure. (That is, if $|\sigma|<f(s)$, then we have determined whether or not $n\in\widetilde A$ for every $n$ in any $P_{\sigma,m,\tau,j,k}$.) We also have some finite set of numbers $n$ such that we maintain a restriction saying that for any $n$ in that set, any deduction procedure that operates under false assumptions about any $b_\sigma$ with $|\sigma|=n$ does not produce any answers. We have not determined the values of $\widetilde A$ on any other deduction procedures.

At this point, we have one single functional, $\varphi_s$, that we need to diagonalize against. This means that we either must ensure that there will be some generic oracle $(A)$ for $A$ such that $\varphi_s^{(A)}$ does not have density-1 domain, or we must ensure that there is some generic oracle $(A)$ for $A$ such that $\varphi_s^{(A)}$ incorrectly computes $B$ at some number.

The first question that we ask is whether there any way of extending our definition of $\widetilde A$ to make $\varphi_s^{A}$ produce an incorrect computation for $B$. (We remind the reader that $T$ is fixed, and $B$ depends entirely on $A$, so determining the value of $\widetilde A$ also determines the values of $A$ and $B$, and so determines whether or not $\varphi_s^{A}$ produces any incorrect computations for $B$. Also, we only need to diagonalize against more-is-more reductions, so we may assume that having more information about the oracle always produces more outputs, and never produces different outputs, and thus, if any generic oracle for $A$ produces incorrect results, then the full oracle for $A$ produces incorrect results.)

If $\varphi_s^{A}$ can be made to produce any incorrect computation for $B$, then we make that extension, and this guarantees that when we are done constructing $\widetilde A$, (and therefore $A$ and $B$,) it will not be true that $A\geq_gB$. After that, we extend $\widetilde A$ arbitrarily in order to meet the hypotheses on what the construction should look like at the beginning of a stage. (This involves finding the largest number $c$ such that $b_\sigma$ is defined for a $\sigma$ with $|\sigma|=c$, or such that some deduction procedure for such a $\sigma$ is defined somewhere, and extending the definition of $\widetilde A$ arbitrarily to all other $b_\tau$'s and deduction procedures for $b_\tau$'s of equal or lesser length.)

If $\varphi_s^A$ cannot be made to produce any incorrect computations for $B$, then we extend our definition of $\widetilde A$ such that for every $\sigma$ with $|\sigma|=f(s)$, $b_\sigma=0$. More importantly, we restrict every deduction procedure that computes such a $b_\sigma$ and that operates under false assumptions to not give any answers. (i.e., to be empty.) We allow all the correct deduction procedures to give correct answers immediately. Finally, we insist that for the rest of the construction, for every deduction procedure that operates under a false assumption about the value of $b_\sigma$ for any $\sigma$ with $|\sigma|=f(s)$, that deduction procedure does not give any answers.

This completes the construction. (The second option will certainly happen infinitely often, and so $\widetilde A$ will be defined everywhere after $\omega$ many steps.)

\vspace{5pt}

\textbf{Part 3}

\vspace{5pt}

Finally, we prove that if $T$ is ill-founded, then $A\ngeq_gB$. To do this, we must demonstrate that for any $\varphi$ that had not been able to be extended to make a false claim about $B$, there exists a generic oracle, $(A)$, for $A$ such that $\varphi^{(A)}$ does not have density-1 domain. (At the end of Part 1, we verified that if $T$ is well-founded, then $A\geq_gB$. Also, if $\varphi$ had been able to be extended to make a false claim about $B$, we would have done that, and then it would certainly not witness $A\geq_gB$.)

We remind the reader that in this part of the proof, we are allowed to work omnisciently, and in particular we will be allowed to know every choice that was made in Part 2, and also to know an example of an infinite path through $T$.

Assume $T$ is ill-founded. Let $Q$ be an infinite path through $T$. Assume that at stage $s$ of the construction $\varphi_s^A$ could not have been made to produce any incorrect computations for $B$.

Let $s_0=s$, and for each $i>0$, let $s_i$ be the $i$th stage $t$ after stage $s$ such that $\varphi_t^A$ was not able to be made to produce any incorrect computations for $B$ at stage $t$. Let $\sigma^i$ be the initial segment of $Q$ such that $|\sigma^i|=s_i$. 

Now, we define $(A)$ to be the oracle for $A$ that does not give answers on the coding locations of any of the $b_{\sigma^i}$, and also does not include the answers (the 1's) from any of the deduction procedures used for deducing $b_{\sigma^0}$, or from any of the deduction procedures used for deducing $b_{\sigma^i}$, unless those deduction procedures have sufficiently large $m$ values that they depend on $b_{\sigma^{i-1}}$.

Then we claim that $(A)$ is a generic oracle for $A$, and that $\varphi_s^{(A)}$ does not have density-1 domain.

To show that $(A)$ is a generic oracle for $A$, we show that the set of places where $(A)$ does not give answers is density-0.

There are finitely many (in fact, at most one) $\sigma^i$ of each possible length. Also, for each length $n$, there is at most one deduction procedure whose answers we erased whose shortest queried string has length $n$. This is because only the correct deduction procedures give any answers that need to be erased, and because we only erase deduction procedures that do not require knowledge of the previous $b_{\sigma^i}$. Therefore the pieces of information that we excluded from $(A)$ were coded into smaller and smaller portions of the corresponding $P_j$. Thus, for every $m$, there is a last $j$ such that the last $\frac1{2^m}$ of $P_j$ was excluded from $(A)$, and after that point, the density of the domain of $(A)$ never again drops below $\frac1{2^{m-1}}$. (Also, it goes above $\frac1{2^{m-1}}$ by the end of $P_{j+1}$.)

Finally, we show that $\varphi_s^{(A)}$ does not give outputs in any of the locations where the $b_{\sigma^i}$ are coded in $B$.

The proof of this fact is that, for any $i$, any finite subset of the information in $(A)$ is a partial oracle that could be extended to a partial oracle for a different set $A_1$ which would also be consistent with the requirements imposed at the beginning of stage $s$ of the construction, and such that the value of $b_{\sigma^i}$ in $A_1$ was different from the value of $b_{\sigma^i}$ in $A$. (This statement will be proved by induction shortly.) Therefore, if  $\varphi_s^{(A)}$ gives any outputs on any of the the $b_{\sigma^i}$, then at stage $s$, we would have been able to extend our condition on $A$ to a condition specifying enough of $A_1$ to cause $\varphi_s$ to produce an incorrect computation for $B$, contradicting our assumption on $\varphi_s$.

\vspace{5pt}

We conclude our proof by proving, by induction on $i$, that any finite subset of the information in $(A)$ is a partial oracle that could be extended to a partial oracle for a different set $A_1$ which would also be consistent with the requirements imposed at the beginning of stage $s$ of the construction, and such that the value of $b_{\sigma^i}$ in $A_1$ was different from the value of $b_{\sigma^i}$ in $A$.

Recall that by the construction of $\widetilde A$, and by the assumption on $\sigma^i$, for every $i$, $b_{\sigma^i}=0$

\vspace{5pt}

Case 1: $i=0$

Assume that we have seen a finite number of the pieces of information in $(A)$. Then we have not seen any of the locations where $b_{\sigma^0}$ is coded, and we have also seen nothing but 0's from all of the deduction procedures for computing $b_{\sigma^0}$. At stage $s$, the value of $b_{\sigma^0}$ had not yet been decided, and there were no conditions yet on the deduction procedures for computing $b_{\sigma^0}$, except for the requirements that certain deduction procedures that operated under false assumptions were not allowed to give any outputs. Furthermore, none of the deduction procedures operating under false assumptions about $b_{\sigma^0}$ ever give answers.

Therefore, it would be consistent with what we have seen so far of $(A)$ and with the requirements imposed at the beginning of stage $s$ to have $b_{\sigma^0}$ equal to 1, and then to fill in $A_1$ with $b_{\sigma^0}=1$, and with the deduction procedures that compute $b_{\sigma^0}$ having 1's in the correct locations in the next coding sets that we have not yet looked at.

Since none of the deduction procedures that operated under the assumption $b_{\sigma^0}=1$ have given answers in the finite amount of $(A)$ that we have seen, we may freely extend them in $A_1$ to give correct outputs in $A_1$. (Notice that this is consistent with the conditions imposed on the construction at the beginning of stage $s$, because the condition that those deduction procedures never give outputs was imposed at the end of stage $s$.)

Case 2: $i>0$

Assume again that we have seen a finite number of the pieces of information in $(A)$. Then, again, we have not seen any of the locations where $b_{\sigma^i}$ is coded, and we have also seen nothing but 0's from all except a few of the deduction procedures for computing $b_{\sigma^i}$. Among these deduction procedures, the only ones from which we have seen any 1's are those which operate under true assumptions, and which also depend on the value of $b_{\sigma^{i-1}}$.

By induction, it is consistent (both with what we have seen of $(A)$, and with the condition on the construction at stage $s$) that $b_{\sigma^{i-1}}$ could have the opposite of its actual value. Thus, it would be consistent to fill in $A_1$ to have the incorrect value for $b_{\sigma^{i-1}}$, and then to also have the incorrect value for $b_{\sigma^i}$, and then to have the ``relevant" deduction procedure then place 1's into the next relevant location. (Here, relevant means operating under the assumptions that are true of $A_1$, but potentially false of $A_0$.)

Again, this does not contradict any of our requirements, since none of the deduction procedures that use $b_{\sigma^i}=1$ have given any outputs yet, so the ones that operate under assumptions that are correct in $A_1$ can still be made to give outputs which are correct in $A_1$.

This concludes our proof of $\mathbf{\Pi_1^1}$-completeness of $\geq_g$.

\end{proof}

We briefly mention one thing about this proof.

The fact that a deduction procedure operating under false assumptions about $b_{\sigma^i}$ never gives incorrect answers is very important. This is what allows us to ensure that it is consistent with any finite oracle that $b_{\sigma^i}$ has the opposite value. However, it is also important that those deduction procedures never give correct answers, because otherwise it might be possible to deduce the value of $b_{\sigma^{i+1}}$ without knowing the value of $b_{\sigma^i}$, simply by seeing every possible deduction procedure for $b_{\sigma^{i+1}}$ give the same answer. It is also important that sometimes, incorrect deduction procedures can give correct answers, because otherwise, seeing a deduction procedure give a correct answer would tell us that all of its assumptions were correct, and it is also important that sometimes incorrect deduction procedures give incorrect answers, because otherwise, seeing any answer from any deduction procedure would be sufficient to know that the answer was correct. By using all three sorts of deductions (incorrect, correct, and nonresponsive), we have enough leeway to ensure that if the opponent is clever enough to avoid making mistakes, then we can force him to not give answers.



\end{document}